\documentclass[a4paper,reqno]{amsart}

\usepackage[english]{babel}
\usepackage{amssymb,amsmath}
\usepackage{amsrefs}
\usepackage[foot]{amsaddr}
\usepackage{transfer}
\numberwithin{equation}{section}
\numberwithin{thm}{section}

\def\HBU{\textup{\textsf{HBU}}}

\begin{document}
\title{Reverse Mathematics and parameter-free Transfer}

\author{Benno van den Berg}
\address{Institute for Logic, Language, and Computation, Universiteit van Amsterdam, The Netherlands}
\email{bennovdberg@gmail.com}
\author{Sam Sanders}
\address{Department of Mathematics, TU Darmstadt, Germany \& School of Mathematics, University of Leeds, UK}
\email{sasander@me.com}
\begin{abstract}
Recently, conservative extensions of Peano and Heyting arithmetic in the spirit of Nelson's axiomatic approach to Nonstandard Analysis, have been proposed.
In this paper, we study the \emph{Transfer} axiom of Nonstandard Analysis restricted to formulas \emph{without} parameters.
Based on this axiom, we formulate a \emph{base theory} for the \emph{Reverse Mathematics} of Nonstandard Analysis and prove some natural reversals, 
and show that most of these equivalences do not hold in the absence of parameter-free \emph{Transfer}.    
\end{abstract}
%

\maketitle


\vspace{-0.7cm}
\section{Introduction}\label{intro}
Recently, conservative extensions of Peano and Heyting arithmetic based on Nonstandard Analysis have been introduced (\cites{benno,fega}).
Use is made of Nelson's \emph{axiomatic} approach to Nonstandard Analysis, pioneered in \emph{internal set theory} (\cite{wownelly}).  In Nelson's framework, the language is extended with a new
unary predicate `st$(x)$', read as `$x$ is standard', governed by three new axioms, namely \emph{Idealization}, \emph{Transfer}, and \emph{Standardization}.
In the setting of \cite{benno}, \emph{Transfer} corresponds to:
\be\label{fulltransfer}
(\forall^{\st}\underline{t})\big[(\forall^{\st}\underline{x})\varphi(\underline{t},\underline{x})\di (\forall \underline{x})\varphi(\underline{t},\underline{x})\big] \tag{$\TPA$}
\ee
where $\varphi(\underline{t},\cdot)$ is internal, i.e.\ without the new predicate `st', and involves no parameters but the (standard) ones shown, namely $\underline{t}$.

\medskip

The authors of \cite{benno} conjectured that adding \ref{fulltransfer} to their version of Peano arithmetic results in a conservative extension (See \cite{benno}*{p.\ 1992}).
Although this conjecture turns out to be wrong in general as shown in Section \ref{conj} (and previously in \cite{fega}), disallowing the parameters $\underline{t}$ in \ref{fulltransfer} \emph{does} result in the envisioned conservative extension.
In particular, consider the following version of \ref{fulltransfer}:
\be\label{pfreetransfer}
(\forall^{\st}\underline{x})\varphi(\underline{x})\di (\forall \underline{x})\varphi(\underline{x}) \tag{$\PFTPA$},
\ee
where the internal formula $\varphi(\underline{x})$ does not involve parameters, i.e.\ all variables are shown.
In Section \ref{param}, we show that \ref{pfreetransfer} gives rise to a conservative extension of various systems of arithmetic.
In this way, we partially answer a question by Avigad from \cite{avi3}*{p.\ 39}, namely `how much \emph{Transfer}' one can conservatively add.  

\medskip

Furthermore, we establish in Section \ref{mikeh} that \ref{pfreetransfer} is extremely useful in proving \emph{Reverse Mathematics}-style equivalences in Nonstandard Analysis.
The reader is referred to \cites{simpson2,simpson1} for an overview of Friedman's program Reverse Mathematics, which we also briefly introduce in Section~\ref{RM}.
In the latter section, we also discuss the results from \cite{dagsam} which suggest that the Reverse Mathematics of Nonstandard Analysis is \emph{fundamentally} different from the classical variety.  
In other words, the Reverse Mathematics of Nonstandard Analysis is a subject in its own right.  

\medskip

We show in Section \ref{tonny} that $\PFTPA$ is essential in establishing some natural equivalences involving comprehension in Nonstandard Analysis, while Section \ref{kext} focusses on the equivalence between internal (i.e.\ not involving Nonstandard Analysis) and external (i.e.\ involving Nonstandard Analysis) statements.
Thus, our results differ substantially from those proved in \cite{keisler1}:  We prove \emph{equivalences} between external and internal statements over a suitable base theory, while Keisler formulates systems of nonstandard arithmetic conservative over the Big Five.  
\section{Preliminaries}
\subsection{Nonstandard Peano arithmetic}\label{base}
In this section, we briefly explain the nonstandard version of Peano Arithmetic as introduced and studied in \cite{benno}.

\medskip

Our starting point will be the system $\epa$ of Peano arithmetic in all finite types, as formalised in \cite[\S 3.3, p.\ 48]{kohlenbach3}.  This is the system called $\epazero$ in \cite{troelstra73} and $\epaarrow$ in \cite{troelstravandalen88b}.
As to some basic properties of $\epa$, only equality of natural numbers is a primitive notion; Equality at higher types is defined extensionally and we have axioms stating that extensional equality is a congruence. 
In addition, our version of $\epa$ does not include product types. 
The price we have to pay for making this choice is that we often end up working with tuples of terms and variables of different types and we will have to adopt some conventions for how these ought to be handled. Fortunately, there are some well-known standard conventions here which we will follow (See \cite{kohlenbach3,troelstra73} or \cite{benno}), as discussed in Notation \ref{klaf} below.

\medskip

The first nonstandard system that we will consider is $\epast$:  The language of $\epast$ is obtained from that of $\epa$ by adding unary predicates $\st^\sigma$ as well as two new quantifiers $\forallst x^\sigma$ and $\existsst x^\sigma$ for every type $\sigma \in \T$. Formulas in the old language of $\epa$, i.e.\ those not containing these new symbols, we will call \emph{internal}; By contrast, general formulas from $\epast$ will be called \emph{external}.  To distinguish clearly between internal and external formulas, we will adopt the following:
\begin{quote}
{\large \sc{Important convention:}} We follow Nelson \cite{wownelly} in using small Greek letters to denote internal formulas and capital Greek letters to denote formulas which can be external.
\end{quote}
Functionals are often assigned Greek letters, but this will not cause confusion.  Furthermore, the (possibly external) $\varphi^{\st}$ is defined from $\varphi$ by appending `st' to all quantifiers (except bounded number quantifiers).

\medskip

The system $\epast$ is $\epa$ plus the basic axioms $\EQ$, $ \Tst$, and  $\IA^{\st{}}$.
\begin{defi}[Basic axioms]\label{defke}~
\begin{itemize}
\item The axiom $\EQ$ stands for the defining axioms of the external quantifiers:
\[
\forallst x \, \Phi(x)  \leftrightarrow   \forall x \, (\, \st(x)\rightarrow\Phi(x) \, ),\textup{ and } \existsst x \, \Phi(x)  \leftrightarrow  \exists x \, (\, \st(x)\wedge\Phi(x) \, ).
\]
\item The axiom $\Tst$ consists of:
\begin{enumerate}
\item The axioms $\st(x) \land x = y \to \st(y)$,
\item The axiom $\st(t)$ for each closed term $t$ in $\T$,
\item The axioms $\st(f)\wedge\st(x)\rightarrow\st(fx)$.
\item The axiom $\st^0(x) \land y \leq x \to \st^0(y)$.
\end{enumerate}
\item The axiom $ \IA^{\st{}}$ is the external induction axiom:
\be\tag{$\IA^{\st{}}$}\label{jax}
\big(\Phi(0)\wedge\forallst x^0 (\Phi(x)\rightarrow\Phi(x+1) )\big)\rightarrow\forallst x^0 \Phi(x).
\ee
\end{itemize}
\end{defi}
In $\EQ$ and $\IA^{\st{}}$, the expression $\Phi(x)$ is an arbitrary external formula in the language of $\epast$, possibly with additional free variables. Besides external induction in the form of $\IA^{\st{}}$, the system $\epast$ also contains the internal induction axiom
\[ \varphi(0) \land \forall x^0 \, ( \, \varphi(x) \to \varphi(x+1) \, ) \to \forall x^0 \, \varphi(x), \]
simply because this is part of $\epa$.  Following our above convention, it is clear that this principle applies to internal formulas only.

\medskip

It is easy to see that $\epast$ is a conservative extension of $\epa$: One gets an interpretation of $\epast$ in $\epa$ by declaring everything to be standard. For more information and results on $\epast$, we refer to \cite{benno, sambon, sambrouw}.
Below, we need the following fragment of the axiom of choice. 
\bdefi\label{QFAC} The axiom $\QFAC$ consists of the following for all finite types $\sigma, \tau$:
\be\tag{$\QFAC^{\sigma,\tau}$}
(\forall x^{\sigma})(\exists y^{\tau})A(x, y)\di (\exists Y^{\sigma\di \tau})(\forall x^{\sigma})A(x, Y(x))
\ee  
\edefi
We finish this section with some remarks on notation.
\begin{nota}[Finite sequences in $\epa$]\label{klaf}\rm
An important notational matter is that inside $\epa$, we should be able to talk about finite sequences of objects of the same type, \emph{not} to be confused with the metalinguistic notion of tuple we mentioned earlier.
There are at least two ways of approaching this:  First of all, as in \cite{benno}, we could extend $\epa$ with types $\sigma^*$ for finite sequences of objects of type $\sigma$, add constants for the empty sequence and the operation of prepending an element to a sequence, as well as a list recursor satisfying the expected equations.  

\medskip

Secondly, as in \cite{bennoandeyvind}, we could exploit the fact that one can code finite sequences of objects of type $\sigma$ as a single object of type $\sigma$ in such a way that every object of type $\sigma$ codes a sequence. Moreover, the standard operations on sequences, such as extracting their length or concatenation, are given by terms in G\"odel's $T$. We choose the second option: this makes scaling down $\epa$ to weaker systems easier.

\medskip

In fact, finite sequences are really stand-ins for finite sets in our setting, and we will often use set-theoretic notation for this reason, i.e.\ $\emptyset$ for the code of the empty sequence, $\cup$ for concatenation and $\{ x \}$ for the finite sequence of length $1$
with sole component $x$. For $x$ and $y$ of the same type we will write $x \in y$ if $x$ is equal to one of the components of the sequence coded by $y$.  Furthermore, for a sequence $\alpha^{0\di \rho}$ and $k^{0}$, the finite sequence $\overline{\alpha}k$ is exactly $\langle \alpha(0), \alpha(1), \dots, \alpha(k-1) \rangle$.  
Finally, if $Y$ is of type $\sigma \to \tau$ and $x$ is of type $\sigma$ we define $Y[x]$ of type $\tau$ as $Y[x] := \cup_{ f \in Y} f(x)$.
\end{nota}

\begin{nota}[Equality]\label{equ}\rm
The system $\textsf{E-PA}^{\omega}$ includes equality between natural numbers `$=_{0}$' as a primitive.  Equality `$=_{\tau}$' and inequality $\leq_{\tau}$ for $x^{\tau},y^{\tau}$ is:
\be\label{aparth}
[x=_{\tau}y] \equiv (\forall z_{1}^{\tau_{1}}\dots z_{k}^{\tau_{k}})[xz_{1}\dots z_{k}=_{0}yz_{1}\dots z_{k}],
\ee
\be\label{aparth1}
[x\leq_{\tau}y] \equiv (\forall z_{1}^{\tau_{1}}\dots z_{k}^{\tau_{k}})[xz_{1}\dots z_{k}\leq_{0}yz_{1}\dots z_{k}],
\ee
if the type $\tau$ is composed as $\tau\equiv(\tau_{1}\di \dots\di \tau_{k}\di 0)$.
In the spirit of Nonstandard Analysis, we define `approximate equality $\approx_{\tau}$' as follows (with the type $\tau$ as above):
\be\label{aparth2}
[x\approx_{\tau}y] \equiv (\forall^{\st} z_{1}^{\tau_{1}}\dots z_{k}^{\tau_{k}})[xz_{1}\dots z_{k}=_{0}yz_{1}\dots z_{k}]
\ee
As suggested by the `\textsf{E}', $\textsf{E-PA}^{\omega}$ includes the \emph{axiom of extensionality} as follows:
\be\label{EXT}\tag{\textsf{E}}  
(\forall  x^{\rho},y^{\rho}, \varphi^{\rho\di \tau}) \big[x=_{\rho} y \di \varphi(x)=_{\tau}\varphi(y)   \big].
\ee
However, the so-called axiom of \emph{standard} extensionality \eqref{EXT}$^{\st}$, defined as:
\[
(\forall^{\st}  x^{\rho},y^{\rho}, \varphi^{\rho\di \tau}) \big[x\approx_{\rho} y \di \varphi(x)\approx_{\tau}\varphi(y)   \big],
\]
is not included in the systems from \cite{benno}.  The reason is that term extraction as in Theorem \ref{consresultcor} does not hold in the presence of \eqref{EXT}$^{\st}$ (See \cite{benno}*{p.\ 1973}).
\end{nota}

\subsection{Introducing Reverse Mathematics}\label{RM}
Reverse Mathematics (RM) is a program in the foundations of mathematics initiated around 1975 by Friedman (\cites{fried,fried2}) and developed extensively by Simpson (\cite{simpson2}) and others.  

\medskip
  
The aim of RM is to find the \emph{minimal} axioms needed to prove a statement of \emph{ordinary} mathematics, i.e.\ dealing with countable or separable objects.   
The classical\footnote{In \emph{Constructive Reverse Mathematics} (\cite{ishi1}), the base theory is based on intuitionistic logic.} \emph{base theory} $\RCA_{0}$ of `computable\footnote{The system $\RCA_{0}$ consists of induction $I\Sigma_{1}$, and the {\bf r}ecursive {\bf c}omprehension {\bf a}xiom $\Delta_{1}^{0}$-CA.} mathematics' is always assumed.
Thus, the aim of RM may be described as follows:  
\begin{quote}
\emph{The aim of \emph{RM} is to find the minimal axioms $A$ such that $\RCA_{0}$ proves $ [A\di T]$ for statements $T$ of ordinary mathematics.}
\end{quote}
Surprisingly, once the minimal axioms $A$ have been found, we almost always have $\RCA_{0}\vdash [A\asa T]$, i.e.\ not only can we derive the theorem $T$ from the axioms $A$ (the `usual' way of doing mathematics), we can also derive the axiom $A$ from the theorem $T$ (the `reverse' way), and hence the name `Reverse Mathematics'.  \medskip

Perhaps even more surprisingly, in the majority\footnote{Exceptions are classified in the so-called Reverse Mathematics zoo (\cite{damirzoo}).
} 
of cases, for a statement $T$ of ordinary mathematics, either $T$ is provable in $\RCA_{0}$, or the latter proves $T\asa A_{i}$, where $A_{i}$ is one of the logical systems $\WKL_{0}, \ACA_{0},$ $ \ATR_{0}$ or $\FIVE$.  The latter together with $\RCA_{0}$ form the `Big Five' and the aforementioned observation that most mathematical theorems fall into one of the Big Five categories, is called the \emph{Big Five phenomenon} (\cite{montahue}*{p.\ 432}).  
Furthermore, each of the Big Five has a natural formulation in terms of (Turing) computability (See e.g.\ \cite{simpson2}*{I.3.4, I.5.4, I.7.5}).
As noted by Simpson in \cite{simpson2}*{I.12}, each of the Big Five also corresponds (sometimes loosely) to a foundational program in mathematics.  

\medskip

Furthermore, RM is inspired by constructive mathematics, in particular the latter's \emph{Brouwerian counterexamples} (See \cite{mandje2} for the latter).  However, in contrast to practice of adding `extra data' to obtain constructive theorems, RM studies mathematical theorems `as they stand', according to Simpson (\cite{simpson2}*{I.9}).  
However, the logical framework for RM is \emph{second-order arithmetic}, i.e.\ only natural numbers and sets thereof are available.  For this reason higher-order objects such as continuous real functions and topologies are not available directly, and are represented by so-called \emph{codes} (See e.g.\ \cite{simpson2}*{II.6.1} and \cite{mummy}), while \emph{discontinuous} functions are not available \emph{tout court}.  
Kohlenbach shows in \cite{kohlenbach4}*{\S4} that the use of codes to represent continuous functions in RM entails a \emph{slight} constructive enrichment, namely a modulus of pointwise continuity.    
He has also introduced \emph{higher-order} RM in which discontinuous functions \emph{are} present (See \cite{kohlenbach2}*{\S2})

\medskip

Next, we consider an interesting observation regarding the Big Five systems of Reverse Mathematics, namely that these five systems satisfy the strict implications:
\be\label{linord}
\FIVE\di \ATR_{0}\di \ACA_{0}\di\WKL_{0}\di \RCA_{0}.
\ee
By contrast, there are many incomparable logical statements in second-order arithmetic.  For instance, a regular plethora of such statements may be found in the \emph{Reverse Mathematics zoo} in \cite{damirzoo}.  The latter is intended as a collection of theorems which fall outside of the Big Five classification of RM.  

\medskip

Finally, it should be noted that the elegant picture as in \eqref{linord} does not extend to the RM of Nonstandard Analysis, as shown in \cite{dagsam}.  
Indeed, let  $\STP$, $\paai$, and $\Paai$ be the nonstandard counterparts of $\WKL_{0}$, $\ACA_{0}$, and $\FIVE$, as introduced in Section \ref{mikeh}. 
In contrast to \eqref{linord}, we have
\be\label{nolinord}
\Paai \not\di \STP \textup{ and } \paai \not \di \STP.  
\ee
over any reasonable base theory.   Similar results hold for $\textsf{LMP}$, the nonstandard counterpart of $\WWKL_{0}$ from \cite{pimpson}*{X.1}.
As shown in \cite{dagsam}, the non-implications from \eqref{nolinord} can be reformulated in terms of non-computability results \emph{not involving Nonstandard Analysis}, i.e.\ computability theory and Nonstandard Analysis turn out to be intimately connected.  
In fact, $\STP$ translates to the \emph{special fan functional} (See Section \ref{kext}) in computability theory, which simply computes a finite cover for Cantor space for \emph{any} (possibly discontinuous) input, while Tait's \emph{fan functional} computes such a cover \emph{only for continuous inputs}.  
Note that Kohlenbach argues in favour of the study of discontinuous functionals in his higher-order RM in \cite{kohlenbach2}.

\section{Transfer}
In this section, we prove our main result regarding \emph{Transfer}, namely that the latter limited to formulas \emph{without parameters} leads to a conservative extension of the original system from \cite{benno} based on $\epa$.  
The same holds for fragments of the latter, and we will establish in Section \ref{mikeh} that parameter-free \emph{Transfer} is an essential part of the RM of Nonstandard Analysis.  

\subsection{On a conjecture regarding Transfer}\label{conj}
In this section, we prove that the conjecture from \cite{benno}*{p.\ 1997} is incorrect (which can also be found in \cite{fega}).  First, we recall the principles $\I$, $\HACint$ and $\TPA$ from \cite{benno}. The former is a typed version of Nelson's idealisation principle
\be
 \forallst y' \, \exists x \, \forall y \in y' \, \varphi(x, y) \to \exists x \, \forallst y \, \varphi(x, y). \tag{\I}
 \ee
Furthermore, $\HACint$ is a weak (`Herbrandized') version of the axiom of choice for internal formulas:
\be
\quad  \quad \forallst x^{\sigma} \, \existsst y^{\tau} \, \varphi(x, y) \to \existsst Y^{\sigma\di \tau} \, \forallst x^{\sigma} \, \exists y^{\tau} \in Y[x] \, \varphi(x, y). \tag{\HACint}
\ee
The term `Herbrandized' is meant to indicate that the choice function $Y$ does not provide a single witness, but only a finite list of candidates;  This list contains an actual witness, but one may not be able to (effectively) find one.  
This setup is similar to the idea of a \emph{Herbrand disjunction}, hence the name. Note that $\HACint$ could also have been formulated using the normal application operation:
\be
 \forallst x \, \existsst y \, \varphi(x, y) \to \existsst Y \, \forallst x \, \exists y \in Y(x) \, \varphi(x, y). \tag{\HACint}
\ee
The two versions are clearly equivalent.  Finally, $\TPA$ is the Transfer principle:
\be\tag{\TPA}
 \forallst \tup t \, ( \, \forallst x \,  \varphi(x, \tup t) \to \forall x \, \varphi(x, \tup t) \, )
 \ee
where `$\forallst \tup t$' is supposed to quantify away all the remaining free variables in $\varphi(x, \tup t)$.

\medskip

It is one of the main results of \cite{benno} that the system $\epast + \I + \HACint$ is conservative over $\epa$. As noted in the introduction, it was conjectured in \cite{benno} that this was not the strongest result possible, and that even $\epast + \I + \HACint + \TPA$ is conservative over $\epa$. This conjecture is unfortunately not true, as was independently observed by the two authors of this paper.  
To see this, consider 
\be\label{WRONG}
\forall x^{0} \, \exists y^{0} \, \varphi(x,y) \di \exists f^{1} \, \forall x^{0} \, \exists y \in f(x) \, \varphi(x,y),
\ee
where $\varphi$ is an internal formula containing only the variables $x$ and $y$ free.
This statement, with all shown quantifiers relativized to `st', easily follows from $\HACint$, and applying $\TPA$ immediately yields \eqref{WRONG}.
However, we have the following lemma.
\begin{lem}\label{poosset} There is a formula $\varphi$ with two free variables $x$ and $y$ such that
\[
\epa \not\vdash \forall x^{0} \, \exists y^{0} \, \varphi(x,y) \di \exists f^{1} \, \forall x^{0} \, \exists y \in f(x) \, \varphi(x,y).
\]
\end{lem}
\begin{proof}
Suppose the statement \eqref{WRONG} would be provable in $\epa$; then it would hold in the HEO-model of $\epa$ (See e.g. \cite{troelstra73}). That is, we fix the hereditarily recursive functions as a model of $\epa$.
Now apply \eqref{WRONG} to
\[
\varphi_{0}(x,y) :\equiv (\forall e^{0}\leq x)\big[(\exists z^{0})(\{e\}(x)=z)\di y>\{e\}(x)  \big].
\]
Here, $\{e\}(x)$ is the value of the $e$-th partial recursive function at $x$, if it exists.
The statement $\forall x^{0} \, \exists y^{0} \, \varphi_{0}(x,y) $ is provable in $\epa$, so \eqref{WRONG} would provide us with a function $f$ such that $\forall x^{0} \, \exists y \in f(x) \, \varphi_{0}(x,y)$.
Now define $g(x) = \max\{y  : y \in f(x) \}$ and note that $g$ grows faster than any total recursive function.
However, such functions do not exist in the HEO-model and we obtain a contradiction.
\end{proof}
In conclusion, the conjecture from \cite{benno}*{p.\ 1997} is wrong in light of Lemma \ref{poosset}. 

\subsection{Parameter-free Transfer}\label{param}
In the previous section, we observed that the conjecture that $\epast + \I + \HACint + \TPA$ is  conservative over $\epa$ is false. Each false conjecture has a silver lining, however, and a slightly modified version of the conjecture, suggested by the second author, \emph{is} true.  Indeed, instead of the `full' Transfer principle $\TPA$, consider the \emph{parameter-free} version
\be\tag{\PFTPA }
 \forallst \tup x \,  \varphi(\tup x) \to \forall \tup x \, \varphi(\tup x)
\ee
where $\varphi(\tup x)$ is not supposed to have any free variables besides those in $\tup x$. In this section, we prove that $\epast + \I + \HACint + \PFTPA$ \emph{is} conservative over $\epast$ using techniques similar to those in \cite{benno}.  In Section \ref{mikeh}, we show that $\PFTPA$ is highly useful for the Reverse Mathematics of Nonstandard Analysis.   

\medskip

In particular, we make use of the $\Sh$-interpretation for $\epast$ as introduced in \cite{benno}*{\S7}. In the latter, the fact was used that in $\epast$, all the logical connectives can be defined using only $\lnot, \lor, \st, \forall$.  Thus, one only has to provide the clauses for the $\Sh$-interpretation of these four connectives as in \cite{benno}*{p.\ 1989}. Of course, this makes for a slick proof of the soundness of the $\Sh$-interpretation, but to compute the $\Sh$-interpretation of a concrete formula, one really does not want to first rewrite the formula using this limited set of connectives, as practice shows that this rarely yields a manageable result.

\medskip

For this reason, in this paper, we also regard $\land, \forallst, \existsst, \exists$ as primitive and additionally determine suitable $\Sh$-interpretations for these. We will only regard implication as a defined connective, with $\Phi \to \Psi$ defined as $\lnot \Phi \lor \Psi$. As we will see, this approach does turn out to be feasible.

\medskip

We now extend the aforementioned $\Sh$-interpretation to the other connectives.

\medskip

First of all, if we define $\Phi \to \Psi$ as $\lnot \Phi \lor \Psi$, this means that, if $\forallst \tup x \, \existsst \tup y \, \varphi(\tup x, \tup y, \tup a)$ is the $\Sh$-interpretation of $\Phi(\tup a)$ and $\forallst \tup u \, \existsst \tup v \, \psi(\tup u, \tup v, \tup b)$ is the $\Sh$-interpretation of $\Psi(\tup b)$, then we have:
\[
\big( \, \Phi(\tup a) \to \Psi(\tup b) \, \big)^\Sh \quad :\equiv \quad \forallst \tup Y, \tup u \, \existsst \tup x, \tup v \, \big( \, \exists \tup y \in \tup Y[\tup x] \, \varphi(\tup x, \tup y, \tup a) \to \psi(\tup u, \tup v, \tup b) \, \big),
\]
modulo some applications of classical logic in the internal matrix, of course.

\medskip

Furthermore, the clauses for the Krivine negative translation $\Phi^\kr :\equiv \lnot \Phi_\kr$ for the extended language are:
\begin{eqnarray*}
\big(\Phi\land\Psi\big)_\kr  &:\equiv & \Phi_\kr\lor \Psi_\kr,\\
\big(\forallst z \, \Phi(x)\big)_\kr  & :\equiv  & \existsst z \, \Phi_\kr(z),\\
\big(\existsst z \, \Phi(x)\big)_\kr  & :\equiv  & \forallst z \, \lnot\lnot\Phi_\kr(z),\\
\big(\exists z \, \Phi(z)\big)_\kr  & :\equiv  & \forall z \, \lnot\lnot\Phi_\kr(z),\\
\end{eqnarray*}
This means that, if we would define:
\begin{eqnarray}
(\Phi(\tup a) \land \Psi(\tup b))^\Sh & :\equiv & \forallst \tup x, \tup u \, \existsst \tup y, \tup v \, \varphi(\tup x, \tup y, \tup a) \land \psi(\tup u, \tup v, \tup b) \notag\\
(\forallst z \, \Phi(\tup a, z))^\Sh & :\equiv & \forallst \tup x, z' \, \existsst \tup y' \, \forall z \in z' \, \exists \tup y \in \tup y' \varphi(\tup x, \tup y, z, \tup a),\label{dokio} \\
(\existsst z \, \Phi(\tup a, z))^\Sh & :\equiv & \forallst \tup X \, \existsst \tup Y, z \, \forall \tup x \in \tup X[z, \tup Y] \, \exists \tup y \in \tup Y[\tup x] \varphi(\tup x, \tup y, z, \tup a),\notag \\
(\exists z \, \Phi(\tup a, z))^\Sh & :\equiv & \forallst \tup X \, \existsst \tup Y \, \exists z \, \forall \tup x \in \tup X[\tup Y] \, \exists \tup y \in \tup Y[\tup x] \, \varphi(\tup x, \tup y, z, \tup a), \notag
\end{eqnarray}
then the soundness proof in \cite{benno}*{p.\ 1990} still works. Indeed, we only have to check that \cite{benno}*{Lemma 7.3} extends to $\land$, $\forallst$, $\existsst$ and $\exists$ as defined, and that the negative translation of the $\EQ$-axiom is provable using that very same $\EQ$-axiom in $\ehanststar$.

\medskip

However, we can do better than the above `primitive' interpretation.
Indeed, we add yet another universal quantifier $\forallststar z$ to $\epast$ with clause:
\be\label{dokio2}
(\forallststar z \, \Phi(\tup a, z))^\Sh  :\equiv  \forallst \tup x, z \, \existsst \tup y \, \varphi(\tup x, \tup y, z, \tup a).
\ee
As it turns out, our new simpler interpretation is equivalent to \eqref{dokio}.
\begin{lem}\label{lemmaonexternalunivqf}
The equivalence $ \forallst z \, \Phi(z) \leftrightarrow \forallststar z \, \Phi(z)$ is $\Sh$-interpretable.
\end{lem}
\begin{proof}
To improve readability, we will ignore the fact that we have to work with tuples and drop the underlining. We have to show that the equivalence of (\ref{dokio}) and (\ref{dokio2}), that is,
\[ \forallst x, z' \, \existsst y' \, \forall z \in z' \, \exists y \in  y' \, \varphi( x,  y, z,  a) \leftrightarrow \forallst  x, z \, \existsst y \, \varphi( x,  y, z,  a) \]
is $\Sh$-interpretable. By the soundness theorem for the $\Sh$-interpretation, it suffices to prove this equivalence in $\epast + \I + \HACint$. 

The left-to-right implication follows immediately by choosing $z' = \{ z \}$. For the right-to-left implication we use $\HACint$ to rewrite
\[ \forallst  x, z \, \existsst  y \, \varphi( x,  y, z, a) \]
as
\[ \forallst x \, \existsst Y \, \forallst z \, \existsst  y \in Y[z] \, \varphi( x,  y, z,  a). \]
Then by choosing
\[ y' = \bigcup_{z \in z'} Y[z], \]
we obtain
\[ \forallst  x, z' \, \existsst  y' \, \forall z \in z' \, \exists y \in  y' \, \varphi( x,  y, z, a), \]
as desired.
\end{proof}
In light of the previous lemma, we can use our simpler interpretation \eqref{dokio2} instead of \eqref{dokio}, leading to the following theorem.
\begin{thm}\label{consresult}
Let $\Phi(\tup a)$ be a formula in the language of $\epast$ and suppose $\Phi(\tup a)^\Sh\equiv\forallst \tup x \, \existsst \tup y \, \varphi(\tup x, \tup y, \tup a)$. If $\Delta_{\intern}$ is a collection of internal formulas and
\[ \epast + \I + \HACint + \PFTPA + \Delta_{\intern} \vdash \Phi(\tup a), \]
then \be\label{keaker}
\epa + \Delta_{\intern} \vdash\ \exists \tup t \, \forall \tup a \, \forall \tup x \, \exists \tup y\in \tup t[\tup x]\ \varphi(\tup x,\tup y, \tup a).
\ee
\end{thm}
\begin{proof}
We proceed by induction on the derivation of $\Phi(\tup a)$ in $\epast + \I + \HACint + \PFTPA + \Delta_{\intern}$.
We can re-use a number of results from \cite{benno}.
To this end, consider \cite{benno}*{Theorem 7.7}.  The latter states that whenever $\epast + \I + \HACint +  \Delta_{\intern} \vdash \Phi(\tup a)$, then there are closed terms $\tup t$ in G\"odel's $\T$ such that $\epa \vdash \forall \tup x \, \exists \tup y \in \tup t [\tup x] \varphi(\tup x, \tup y, \tup a)$. Note that $\tup t$ in no way depends on the parameters $\tup a$.
The previous derivation implies in particular that $\epa \vdash \exists \tup t \, \forall \tup a \, \forall \tup x \, \exists \tup y \in \tup t [\tup x] \, \varphi(\tup x, \tup y, \tup a)$. So whenever $\Phi(\tup a)$ is an axiom of $\epast + \I + \HACint$, we have $\epa \vdash\ \exists \tup t \, \forall \tup a \, \forall \tup x \, \exists \tup y\in \tup t[\tup x] \, \varphi(\tup x,\tup y, \tup a)$. Therefore, it only remains to consider the inference rules and the new axiom $\PFTPA$.

\medskip

With regard to the inference rules of $\epast$ in \cite{benno}, we can take a formalisation in which Modus Ponens is the only inference rule. Again, we momentarily ignore the fact that we are working with tuples. 
So if $\Phi(a)^\Sh \equiv \forallst x \, \existsst y \, \varphi(x, y, a)$ and $\Psi(b)^\Sh \equiv \forallst u \, \forallst v \, \psi(u, v, b)$, then we have to prove that from
\[
\epa \vdash \exists  X, V \, \forall a, b \, \forall Y, u \exists x \in X[Y, u] \exists v \in V[Y, u] \, \big( \, \exists y \in Y[x] \, \varphi(x, y, a) \to \psi(u, v, b) \, \big)
\]
and from
\[
\epa \vdash \exists t \, \forall a , x \, \exists y \in t[x] \, \varphi(x, y, a),
\]
follows that
\[
\epa \vdash \exists s \, \forall b,  u \, \exists v \in s[u] \, \psi(u, v, b).
\]
Proving this is easy: Reasoning in $\epa$, let $X, V, t$ be as in the premises and define $s$ such that $s[u] = V[t, u]$.

\medskip

Finally, it remains to consider $\PFTPA$. For this, we need to show that
\[ \epa \vdash \exists t \, \exists y \in t \, \big( \, \varphi(y) \to \forall z \, \varphi(z) \, \big), \]
or
\[ \epa \vdash \exists s \, \big( \, \varphi(s) \to \forall z \, \varphi(z) \, \big). \]
But the latter is a classical tautology, known as the so-called Drinker's Principle.
\end{proof}

\begin{rem}\label{pfessential}\rm
Note that in the last step of the proof, if $\varphi$ had an additional standard parameter $x$, we would need to prove
\be\label{towald}
\epa \vdash \exists t \, \forall x \, \exists y \in t[x] \, \big( \, \varphi(y, x) \to \forall z \, \varphi(z, x) \, \big).
\ee
Here, we would need the provable existence of a kind of `Herbrandized Skolem function'. This seems impossible without some form of choice or comprehension. Thus, the big difference between general and parameter-free transfer seems to derive from the fact that the provable existence of genuine Skolem functions is not so harmless (indeed, it is a kind of axiom of choice), while the provable existence of Skolem constants is just logic.
\end{rem}

\begin{cor}\label{conservativity}
The system $\epast + \I + \HACint + \PFTPA$ is conservative over $\epa$. 
In fact, if we have 
\[ 
\epast + \I + \HACint + \PFTPA \vdash \forallst \tup x \, \exists \tup y \, \varphi(\tup x, \tup y, \tup a), 
\] 
then $\epa \vdash \forall \tup x \, \exists \tup y \, \varphi(\tup x, \tup y, \tup a)$.
\end{cor}
\begin{proof}
Follows from the previous theorem and the fact that $(\forallst \tup x \, \exists \tup y \, \varphi(\tup x, \tup y, \tup a))^\Sh$ is $\forall \tup x \, \exists \tup y \, \varphi(\tup x, \tup y, \tup a)$ for internal $\varphi(\tup x, \tup y, \tup a)$.
\end{proof}
Clearly, the above proof does not need the full strength of Peano arithmetic.  Indeed, careful inspection of the above proofs reveals that instead of Peano Arithmetic, we can use \emph{Elementary Function Arithmetic} ($\efa$), aka $I\Delta_{0}+\usftext{EXP}$.

\medskip

Thus, let $\eefa^{\omega}_{\st}$ be the analogue of $\epast$ with first-order strength $\efa$, similar to Avigad's \textsf{ERA}$^{\omega}$ or Kohlenbach's system \textsf{E-G$_{3}$A$^{\omega}$} (See \cite{avi3}*{p.\ 31} and \cite{kohlenbach3}*{p.\ 55}).
In particular, the system $\eefa^{\omega}_{\st}$ is \textsf{E-G$_{3}$A$^{\omega}$} plus $\EQ$ and {$\Tst$}.  
\begin{cor}\label{conservativity2}
The system $\textup{$\eefa^{\omega}_{\st}$} + \I + \HACint + \PFTPA$ is conservative over $\textup{\eefa$^{\omega}$}$.  If we add $\QFAC^{1,0}$, we obtain a $\Pi_{2}^{0}$-conservative extension of $\efa$.
\end{cor}
The second part (involving $\QFAC^{1,0}$), follows from the results in \cite{kohlenbach2}*{p.\ 293-294}.

\medskip

We shall mostly work over the base theory $\B_{0}\equiv \textup{$\eefa^{\omega}_{\st}$} + \I + \HACint + \PFTPA+\QFAC^{1,0}$ and sometimes use $\B_{0}^{-}$ which is $\B_{0}$ minus $\PFTPA$.   
Due to its central role in Reverse Mathematics, we also list this conservation result for $\pra$.  The system $\epra^{\omega}_{\st}$ is \textup{$\epra^{\omega}$} from \cite{kohlenbach2} plus $\EQ +\Tst$.  

\begin{cor}\label{conservativity3}
The system $\textup{$\epra^{\omega}_{\st}$} + \I + \HACint + \PFTPA$ is conservative over $\textup{$\epra^{\omega}$}$.  If we add $\QFAC^{1,0}$, we obtain a $\Pi_{2}^{0}$-conservative extension of $\pra$.
\end{cor}
The second part (involving $\QFAC^{1,0}$), follows from the results in \cite{kohlenbach2}*{p.\ 293-294}.

\medskip

The previous two corollaries are a considerable strengthening of Avigad's earlier results \cite{avi3}.
In the next sections, we show that $\PFTPA$ is also very useful `in practice', namely in Reverse Mathematics.  However, this `usefulness' comes at a price, as the system $\B_{0}^{-}$ satisfies the following theorem, 
where a \emph{term of G\"odel's system $T$ is obtained}, to be compared to the \emph{existence} of a functional in \eqref{keaker}.
\begin{thm}[Term extraction]\label{consresultcor}
If $\Delta_{\textsf{\textup{int}}}$ is a collection of internal formulas and $\psi$ is internal, and
\be\label{bog}
\B_{0}^{-} + \Delta_{\textsf{\textup{int}}} \vdash (\forall^{\st}\underline{x})(\exists^{\st}\underline{y})\psi(\underline{x},\underline{y}, \underline{a}), 
\ee
then one can extract from the proof closed terms $t$ from G\"odel's $T$ such that
\be\label{effewachten}
\textup{\eefa$^{\omega}$}+\QFAC^{1,0}+ \Delta_{\textsf{\textup{int}}} \vdash (\forall \underline{x})(\exists \underline{y}\in t(\underline{x}))\psi(\underline{x},\underline{y},\underline{a}).
\ee
\end{thm}
\begin{proof}
The proof of \cite{benno}*{Theorem 7.7} goes through for any fragment of Peano arithmetic which includes \textup{\eefa$^{\omega}$}.  
In particular, the exponential function is (all what is) required to `easily' manipulate finite sequences.    
\end{proof}
Finally, we consider the following remark on the practice of $\B_{0}$.
\begin{nota}[Using $\HACint$]\label{efkeseenpaarbladzijdentoevoegen}\rm
As noted in Notation \ref{klaf} and Section \ref{conj}, finite sequences play an important role in our systems.   
In particular, $\HACint$ produces a functional $Y$ which outputs a \emph{finite sequence} of witnesses $Y[x]$.  
However, $\HACint$ provides an actual \emph{witnessing functional} assuming (i) $\tau=0$ in $\HACint$ and (ii) the formula $\varphi$ from $\HACint$ is `sufficiently monotone' as in: 
$(\forall^{\st} x^{\sigma},n^{0},m^{0})\big([n\leq_{0}m \wedge\varphi(x,n)] \di \varphi(x,m)\big)$.    
Indeed, in this case one simply defines $G^{\sigma+1}$ by $G(x^{\sigma}):=\max_{i<|Y(x)|}Y(x)(i)$.  Clearly, we have $(\forall^{\st}x^{\sigma})\varphi(x, G(x))$.  
Similarly, in case $\varphi$ in $\HACint$ is (equivalent to) a quantifier-free internal formula, a finite search (definable in all systems considered in this paper) allows one to select the least $j<|Y(x)|$ such that $\varphi(x, Y(x)(j))$, i.e.\ we also obtain an actual witnessing functional in this case.  
To save space in proofs, we will sometimes skip the (obvious) step involving the maximum of finite sequences (or involving the finite search) when applying $\HACint$.  
We assume the same convention for terms obtained from Theorem~\ref{consresultcor}, and applications of the contraposition of idealisation \textsf{I}.

\end{nota}

\section{Reverse Mathematics and Nonstandard Analysis}\label{mikeh}
We prove a number of results in the RM of Nonstandard Analysis.
In particular, we focus on equivalences for which $\PFTPA$ appears essential.  
Whenever similar or analogous results are readily obtainable, we only indicate this possibility without going into details.  
Our results often make crucial use of the axiom of choice for quantifier-free formulas $\QFAC$, and the role of the latter is
discussed in Remark~\ref{omdageaandringt}.  We assume familiarity with \emph{higher-order} RM and its base theory $\RCAo$ as in \cite{kohlenbach2}.
\subsection{Equivalences involving comprehension}\label{tonny}
We study the natural counterparts in Nonstandard Analysis of some basic equivalences involving comprehension in higher-order RM.  
We also show that the equivalences in Nonstandard Analysis only hold in the presence of $\PFTPA$.  
We shall observe that to guarantee natural equivalences like in Theorem \ref{hanker} and to avoid constantly keeping track of extensionality, $\PFTPA$ seems unavoidable in the RM of Nonstandard Analysis.  
\subsubsection{Arithmetical comprehension}
First of all, we consider the equivalence\footnote{Note that the system $\RCAo$ is only (explicitly) introduced in \cite{kohlenbach2}, and not in \cite{kohlenbach7}.} $(\mu^{2})\asa (\exists^{2})$ over $\RCAo$ from \cite{kohlenbach7}, where the latter functionals are defined as follows:
\be\label{wiske}\tag{$\exists^{2}$}
(\exists \varphi^{2}\big[(\forall f^{1} )(\varphi (f)=0 \asa (\exists x^{0})f(x)=0)\big],
\ee
\be\tag{$\mu^{2}$}
(\exists \mu^{2})\big[(\forall f^{1})\big((\exists x^{0})(f(x)=0)\di f(\mu(f))=0\big)\big].
\ee
The functional $\mu$ is sometimes referred to as `Feferman's non-constructive mu operator'.
Let $\textsf{TJ}(\varphi)$ and $\MU(\mu)$ be the formulas in square brackets in $(\exists^{2})$ and $(\mu^{2})$.   

\medskip

In light of $(\mu^{2})\asa (\exists^{2})$ and $(\mu^{2})^{\st}\asa (\exists^{2})^{\st}$, we also expect to have e.g.\ $(\exists^{\st} \varphi^{2})\TJ(\varphi)\asa (\exists^{\st}\mu^{2})\MU(\mu)$, 
but the latter equivalence requires $\PFTPA$.  
\begin{thm}\label{hanker}
The system $\B_{0}$ proves $(\exists^{\st} \varphi^{2})\TJ(\varphi)\asa (\exists^{\st}\mu^{2})\MU(\mu)$, while the system $\B_{0}^{-}$ does not. 
The system $\B_{0}^{-}$ does prove $(\mu^{2})^{\st}\asa (\exists^{2})^{\st}$.
\end{thm}
\begin{proof}
For the first equivalence, $(\exists^{\st} \varphi^{2})\TJ(\varphi)\leftarrow (\exists^{\st}\mu^{2})\MU(\mu)$ is immediate as the implication in $(\mu^{2})$ is clearly an equivalence.  
For the first forward implication $(\exists^{\st} \varphi^{2})\TJ(\varphi)\di (\exists^{\st}\mu^{2})\MU(\mu)$, note that the antecedent implies the sentence:
\be\label{korfgu}
(\exists \varphi^{2})\big[ \TJ(\varphi)\wedge (\forall g^{1})(\varphi(g)=0\di (\exists m^{0})(g(m)=0))  \big].
\ee
Applying $\QFAC^{1,0}$ to the second conjunct in \eqref{korfgu}, there is $\Phi^{2}$ such that
\be\label{korfgu2}
(\exists \varphi^{2}, \Phi^{2})\big[ \TJ(\varphi)\wedge (\forall g^{1})(\varphi(g)=0\di (g(\Phi(g))=0))  \big],
\ee
and applying $\PFTPA$ to the sentence \eqref{korfgu2}, there are standard such $\varphi ,\Phi$.  
We thus obtain $(\exists^{\st}\mu^{2})\MU(\mu)$ and the first equivalence is established.  

\medskip

For the second equivalence in the theorem, the forward implication is immediate in the same way as for the first equivalence.  
For the second reverse implication, let standard $\varphi^{2}$ be as in $[\TJ(\varphi)]^{\st}$ and note that the latter implies $(\forall^{\st}f^{1})(\exists^{\st} x^{0})(\varphi(f)=0\di f(x)=0)$, and applying $\HACint$, there is standard $\Phi$ such that 
\[
(\forall^{\st}f^{1})(\exists x^{0}\in \Phi[f])(\varphi(f)=0\di f(x)=0), 
\]
and define $\mu(f)$ as $\Phi(f)(i)$ for $i<|\Phi(f)|$ where $i$ is the least number such that $f(\Phi(f)(i))$.  Then $\mu^{2}$ is standard and satisfies $[\MU(\mu)]^{\st}$.  

\medskip

Finally, for the non-implication, 
suppose $\B_{0}^{-}$ proves  $(\exists^{\st} \varphi^{2})\TJ(\varphi)\di (\exists^{\st}\mu^{2})\MU(\mu)$.  The latter implication yields:
\be\label{forgi}
[(\exists^{\st} \varphi^{2})\TJ(\varphi)]\di (\forall^{\st} f^{1})\big((\exists x^{0})(f(x)=0)\di (\exists^{\st}y^{0})(f(y)=0)\big)
\ee
since standard functionals have standard output for standard input.  Bringing all standard quantifiers outside, we obtain
\be\label{averagejoe}
(\forall^{\st}\varphi^{2}, f^{1})(\exists^{\st} y^{0})[\TJ(\varphi)\di \big((\exists x^{0})(f(x)=0)\di f(y)=0\big)].
\ee
Applying Theorem \ref{consresultcor} to `$\B_{0}^{-}\vdash \eqref{averagejoe}$', we obtain a term $t$ from G\"odel's $T$ such that 
\[
(\forall \varphi^{2}, f^{1})(\exists  y^{0}\in t(\varphi, f))[\TJ(\varphi)\di \big((\exists x^{0})(f(x)=0)\di f(y)=0\big)]
\]
is provable in \textup{\eefa$^{\omega}$}.  Define $s(\varphi)(f)$ as $t(\varphi, f)(i)$ where $i<|t(\varphi, f)|$ is the least number such that $f(t(\varphi, f)(i))=0$ if such there is, and zero otherwise.
Clearly, $s$ is a term of G\"odel's $T$ which converts $\varphi^{2}$ as in $(\exists^{2})$ into $\mu^{2}$ as in $(\mu^{2})$.  However, as discussed in \cite{kohlenbach7}*{Remark~3.6}, such a term 
 cannot exist for the following reason: $\varphi^2$ from $(\exists^{2})$ can be chosen to satisfy $ \varphi \leq_{2} 1$, i.e.\ such $\varphi^2$ has a majorant in the full set-theoretic model, as do terms of G\"odel's $T$, while $\mu^2$ as in $(\mu^{2})$ does not, a contradiction.  
This contradiction yields the non-implication.
\end{proof}
In light of the theorem, it is reasonable to include $\PFTPA$ in the base theory $\B_{0}$.  

\subsubsection{Comprehension for $\Pi_{1}^{1}$-formulas and beyond}
Secondly, we show that Theorem \ref{hanker} easily generalises to other functionals, like those embodying $\Pi_{1}^{1}$-comprehension.
\be\label{suske}
(\exists S^{2})\big[(\forall f^{1})(   S(f)=_{0} 0 \asa (\exists g^{1})(\forall x^{0})(f(\overline{g}x)=0)\big] \tag{$S^{2}$}
\ee
\be\label{NU}\tag{$\mu_{1}$}
(\exists \mu_{1}^{1\di 1})\big[(\forall f^{1})\big((\exists g^{1})(\forall x^{0})(f(\overline{g}x)=0)\di (\forall x^{0})(f(\overline{\mu_{1}(f)}x)=0)\big) \big]
\ee
Recall the notation involving the `overline' from Notation \ref{klaf}.  
The functional $(\mu_{1})$ may be found in \cite{avi2}*{\S8} while the \emph{Suslin functional} $(S^{2})$ may be found in \cite{yama1,kohlenbach2}.  Clearly, $\RCAo+\QFAC^{1,1}$ proves $(\mu_{1})\asa (S^{2})$ in the same way as in \cite{kohlenbach7}*{Cor.\ 3.5}.  Let $\SU(S)$ and $\MU_{1}(\mu_{1})$ be the formulas in square brackets in $(S^{2})$ and $(\mu_{1})$.  
\begin{cor}\label{horkings}
$\B_{0}+\QFAC^{1,1}$ proves $(\exists^{\st} S^{2})\SU(S)\asa (\exists^{\st}\mu_{1})\MU_{1}(\mu_{1})$, while $\B_{0}^{-}+\QFAC^{1,1}$ does not.  $\B_{0}^{-}$ does prove $(S^{2})^{\st}\asa (\mu_{1})^{\st}$.  
\end{cor}
\begin{proof}
The first reverse implication is again immediate while for the first forward implication, similar to \eqref{korfgu}, we have
\be\label{korfguf}
(\exists S^{2})\big[ \SU(S)\wedge (\forall f^{1})(S(f)=0\di  (\exists g^{1})\underline{(\forall x^{0}) (f(\overline{g}x)\ne 0))}  \big], 
\ee
in which the underlined formula may be treated as quantifier-free due to $(\exists^{2})$, which follows from $(S^{2})$.   Applying $\QFAC^{1,1}$ to the second conjunct of \eqref{korfguf}, one obtains a version of \eqref{korfgu2}.  Applying $\PFTPA$ to the resulting formula yields $(\exists^{\st}\mu_{1})\MU_{1}(\mu_{1})$.  

\medskip

The second reverse implication follows by noting that the implication in $(\mu_{1})^{\st}$ is actually an equivalence.  
Now, to check whether the formula $(\forall^{\st} x^{0})(f(\overline{\mu_{1}(f)}x)=0)$ in the latter holds, use $ (\mu^{2})^{\st}$, which follows from $(\mu_{1})^{\st}$ by rewriting $(\exists^{\st} n_{0})(f(n_{0}))=0$ as $(\exists^{\st} g_{0}^{1})(\forall^{\st} x)(\tilde{f}(\overline{g_{0}}x)=0)$, 
where $g_{0}^{1}$ is constant $n_{0}$ and $\tilde{f}(\sigma^{0^{*}})=f(\sigma(0))$ is standard if $f$ is.  For the second forward implication, $(S^{2})^{\st}$ implies $\eqref{korfguf}^{\st}$ as in the first part of the proof.  
Since the former axiom also implies $(\exists^{2})^{\st}$, the underlined formula in $\eqref{korfguf}^{\st}$ may be treated as quantifier-free.  Applying $\HACint$ to the second conjunct of $\eqref{korfguf}^{\st}$, we obtain standard $\Phi$ such that
\be\label{nogmeer}
(\forall^{\st} f^{1})(S(f)=0\di  (\exists g^{1}\in \Phi(f))\underline{(\forall^{\st} x^{0}) (f(\overline{g}x)\ne 0))}
\ee
The underlined formula in \eqref{nogmeer} may be treated as quantifier-free due to $(\exists^{2})^{\st}$.  
Hence, following Notation \ref{efkeseenpaarbladzijdentoevoegen}, we can obtain $\mu_{1}$ as in $(\mu_{1})$ from $\Phi$ as in \eqref{nogmeer}.  

\medskip

Finally, for the non-implication, one proceeds analogously to the proof of the theorem:  suppose the system $\B_{0}^{-}+(\exists^{2})+\QFAC^{1,1}$ does prove $(\exists^{\st} S^{2})\SU(S)\di (\exists^{\st}\mu_{1})\MU_{1}(\mu_{1})$ and apply Theorem \ref{consresultcor} to obtain a term $t$ of G\"odel's $T$ 
such that $(\forall S^{1})(\exists \mu_{1}\in t(S))[\SU(S)\di \MU_{3}(\mu_{1})]$.
Note that Notation \ref{efkeseenpaarbladzijdentoevoegen} (the case of the finite search) applies, i.e.\ we can explicitly define $\mu_{1}$ from $(\mu_{1})$ in terms of $t(S)$ where $S$ is the Suslin functional.  
However, following the same reasoning as in the proof of Theorem \ref{hanker} and \cite{kohlenbach7}*{Remark~3.6}, the functional $\mu_{1}$ is not majorisable, while $S^{2}$ can be chosen to satisfy $S\leq_{2}1$, and terms of G\"odel's $T$ are majorisable, a contradiction, and the non-implication follows. 
\end{proof}
The previous corollary is also readily obtained for the functional $(\E_{2})$, which gives rise to full second-order arithmetic. 
\be\label{suske2}
(\exists T^{3})\big[(\forall \varphi^{2})\big(  T(\varphi)=_{0} 0 \asa (\exists g^{1})(\varphi(g)=0)\big)\big]. \tag{$\E_{2}$}
\ee
\be\label{suske3}
(\exists U^{3})\big[(\forall \varphi^{2})\big( (\exists g^{1})(\varphi(g)=0)\di \varphi(U(\varphi))=0\big)\big]. \tag{$U^{3}$}
\ee
The functional $(\E_{2})$ is studied in \cite{hunter} while one readily proves $(U^{3})\asa (\E_{2})$ over $\RCAo+\QFAC^{2,1}$ as for the Suslin functional.  
The functional $(\E_{2})$ is called $(\exists^{3})$ in \cite{dagsam}.
Let $\textsf{SO}(T)$ and $\MU_{3}(U)$ be the formulas in square brackets in $(\E_{2})$ and $(U^{3})$.  
We have the following corollary.    
\begin{cor}\label{crossref}
The system $\B_{0}+\QFAC^{2,1}$ proves $(\exists^{\st} T^{3})\textsf{\textup{SO}}(T)\asa (\exists^{\st}U^{3})\MU_{3}(U)$, while $\B_{0}^{-}+\QFAC^{2,1}$ does not.  The system $\B_{0}^{-}$ does prove $(\E_{2})^{\st}\asa (U^{3})^{\st}$.  
\end{cor}
\begin{proof}
The first reverse implication is immediate.  The first forward implication follows by considering the following variation of \eqref{korfgu} and \eqref{korfguf}:
\be\label{gehebtalhonderd}
(\exists T^{3})\big[\SO(T)\wedge (\forall \varphi^{2})\big(  T(\varphi)=_{0} 0 \di (\exists g^{1})(\varphi(g)=0)\big)\big]. 
\ee
Similar to the proof of the theorem, apply $\QFAC^{2,1}$ to the second conjunct, followed by $\PFTPA$, to obtain $(\exists^{\st}U^{3})\MU_{3}(U)$.

\medskip

The second reverse implication is immediate.  The second forward implication follows by considering $\eqref{gehebtalhonderd}^{\st}$ and using $\HACint$ instead of $\QFAC^{2,1}$. 
Note that Notation \ref{efkeseenpaarbladzijdentoevoegen} (the case of the finite search) applies again.  

\medskip

Finally, the non-implication follows in the same way as for the theorem and the previous corollary: if $B_{0}^{-}+\QFAC^{2,1}$ proves $(\exists^{\st} T^{3})\textsf{\textup{SO}}(T)\di (\exists^{\st}U^{3})\MU_{3}(U)$, 
then Theorem \ref{consresultcor} provides a term $t$ of G\"odel's $T$ such that $(\forall T^{3})(\exists U^{3}\in t(T))(\SO(T)\di \MU_{3}(U))$.
Note that Notation \ref{efkeseenpaarbladzijdentoevoegen} (the case of the finite search) applies, i.e.\ we can explicitly define $U_{0}$ such that $\MU_{3}(U_{0})$ in terms of $t(T)$.  
Thus, there is a term of G\"odel's $T$ expressing the (non-majorisable) functional $U^{3}$ in terms of the (trivially majorisable) functional $T^{3}$.  
As in the previous, this leads to a contradiction. 
\end{proof}

\subsubsection{The role of extensionality}
Thirdly, we show that $\PFTPA$ simplifies the Reverse Mathematics of Nonstandard Analysis by obviating the need to keep track of the axiom of extensionality \eqref{EXT}$^{\st}$.  

\medskip

Now, the attentive reader has noted that the consequent of \eqref{forgi} is a fragment of \emph{Transfer}, namely the following restriction to $\Pi_{1}^{0}$-formulas:
\be\tag{$\paai$}
(\forall^{\st} f^{1})\big((\forall^{\st} x^{0})(f(x)=0)\di (\forall y^{0})(f(y)=0)\big)
\ee
Similar to how one `bootstraps' $\Pi_{1}^{0}$-comprehension to $\ACA_{0}$, the system $\B_{0}^{-}+\paai$ proves $\varphi\asa \varphi^{\st}$ for any internal arithmetical formula (only involving standard parameters).  
By contraposition (using $\HACint$ and the basic axioms), we immediately obtain the following regarding $\paai$ and $(\mu^{2})$:
\be\label{hankio}
(\exists^{\st}\mu^{2})\MU(\mu)\di\paai \di (\exists^{\st}\mu^{2})(\forall^{\st}f^{1})\MU(\mu, f), 
\ee
where $\MU(\mu, f)$ is just $\MU(\mu)$ with the leading quantifier dropped (and the same for $\TJ(\varphi, f)$).  In light of the above and the fact that the final implication in \eqref{hankio} reverses, it is a natural question whether $\paai$ is equivalent to a version of $(\exists^{2})$.  
To answer this question, let $(\textsf{E})_{2}$ be the axiom of extensionality restricted to type two functionals.  We have the following theorem.      
\begin{thm}\label{hollfy}
$\B_{0}^{-}$ proves $\paai\asa[ \textsf{\textup{(E)}}^{\st}_{2}+ (\exists^{\st} \varphi^{2})(\forall^{\st}f^{1})\TJ(f, \varphi)]$; $\B_{0}+\QFAC^{2,0}$ proves $\paai\asa (\exists^{\st} \varphi^{2})(\forall^{\st}f^{1})\TJ(f, \varphi)$; $\B_{0}^{-}+\QFAC^{2,0}$ does not. 
\end{thm}
\begin{proof}
To establish the non-implication, $ (\exists^{\st} \varphi^{2})(\forall^{\st}f^{1})\TJ(f, \varphi)\di \paai$ implies \eqref{averagejoe}, and a contradiction is obtained in the same way for Theorem \ref{hanker}.  

\medskip

For the first forward implication, the axiom of extensionality as in $\eqref{EXT}_{2}$ implies 
\be\label{tex}
(\forall Y^{2}, f^{1}, g^{1})(\exists N^{0})[{f}N=_{0}{g}N\di Y(f)=_{0}Y(g)], 
\ee
and applying (the contraposition of) $\paai$ to \eqref{tex} for standard $Y, f, g$, we obtain $\eqref{EXT}_{2}^{\st}$.  By \eqref{hankio}, $\paai$ also implies $ (\exists^{\st} \varphi^{2})(\forall^{\st}f^{1})\TJ(f, \varphi)$.  

\medskip

For the first reverse implication, assume $ (\exists^{\st} \varphi^{2})(\forall^{\st}f^{1})\TJ(f, \varphi)$ and $\eqref{EXT}_{2}^{\st}$.  
Now suppose $\neg\paai$, i.e.\ there is standard $h_{0}$ such that $(\forall^{\st}n^{0})h(n)=1$ and $ (\exists m^{0})h(m)= 0$.  Clearly, $h\approx_{1}11\dots$, hence $\varphi(h)=\varphi(11\dots)$ by standard extensionality for standard $\varphi$ as in $(\forall^{\st}f^{1})\TJ(f, \varphi)$.  
However $\varphi(h)=0\ne1=\varphi(11\dots) $, by the definition of $\varphi$, and this contradiction establishes the first equivalence.  

\medskip

For the second equivalence, it suffices to prove $\eqref{EXT}_{2}^{\st}$ in $\B_{0}+\QFAC^{2,0}$, in light of the (proof of the) first equivalence.  
To this end, apply $\QFAC^{2, 0}$ to \eqref{tex} to obtain 
\be\label{tex3}
(\exists \Phi^{3})(\forall Y^{2}, f^{1}, g^{1})[{f}\Phi(Y, f,g)=_{0}{g}\Phi(Y, f,g)\di Y(f)=_{0}Y(g)], 
\ee
and apply $\PFTPA$ to obtain standard such $\Phi$, which implies $\eqref{EXT}_{2}^{\st}$.  
\end{proof}
Hence, if one does not wish to constantly keep track of standard extensionality, $\PFTPA$ seems unavoidable in the RM of Nonstandard Analysis.  
Note that Kohlenbach's higher-order RM from \cite{kohlenbach2} adopts the axiom of extensionality.  

\medskip

As it turns out, a similar result holds for $(\E_{2})$ and a fragment of \emph{Transfer}.  
\be\tag{$\SOT$}
(\forall^{\st}\varphi^{2})\big[(\exists f^{1})\varphi(f)=0\di (\exists^{\st}f^{1})\varphi(f)=0    ],
\ee
Let $\SO(T, \varphi)$ be $(\E_{2})$ with the two outermost quantifiers removed and let $\eqref{EXT}_{3}$ be the axiom of extensionality restricted to type three functionals.  
\begin{thm}\label{polipo}
$\B_{0}^{-}$ proves $\SOT\asa [\textsf{\textup{(E)}}^{\st}_{3}+ (\exists^{\st} T^{3})(\forall^{\st}\varphi^{2})\SO(T, \varphi)]$; $\B_{0}+\QFAC^{3,1}$ proves $\SOT\asa (\exists^{\st} T^{3})(\forall^{\st}\varphi^{2})\SO(T, \varphi)$ while $\B_{0}^{-}+\QFAC^{3,1}$ does not. 
\end{thm}
\begin{proof}
For the first forward implication, note that $\eqref{EXT}_{3}$ implies
\be\label{tex2}
(\forall \xi^{3}, Y^{2}, Z^{^{2}})(\exists f^{1})[Z(f)=_{0}Y(f) \di \xi(Z)=_{0}\xi(Y)], 
\ee
and applying (the contraposition of) $\SOT$ to \eqref{tex2} for standard $\xi, Z, Y$, we obtain $\eqref{EXT}_{3}^{\st}$. 
Applying $\HACint$ to $\SOT$ readily yields $(\exists^{\st} T^{3})(\forall^{\st}\varphi^{2})\SO(T, \varphi)$, using the second case in Notation \ref{efkeseenpaarbladzijdentoevoegen}.
For the first reverse implciation, suppose $\SOT$ is false, i.e.\ there is standard $\varphi_{0}^{2}$ such that $(\forall^{\st}f^{1})\varphi_{0}(f)\ne0$ but $(\exists g^{1})\varphi_{0}(g)= 0$.  
Define $\varphi_{1}^{2}$ as follows: $\varphi_1(f):=\varphi_{0}(f)$ if the latter is nonzero, and $1$ otherwise.  Clearly $\varphi_{0}\approx_{2}\varphi_{1}$ but $T(\varphi_{0})=0\ne T(\varphi_{1})$, i.e.\ a contradiction with the right-hand side of the first equivalence ensues.  

\medskip

For the second equivalence, it suffices to prove $\eqref{EXT}_{3}^{\st}$ in $\B_{0}+\QFAC^{3,1}$, in light of the (proof of the) first equivalence.  
To this end, apply $\QFAC^{3, 1}$ to \eqref{tex2} to obtain 
\be\label{tex4}
(\exists \Psi)(\forall \xi^{3}, Y^{2}, Z^{^{2}})[Z(\Psi(\xi, Y, Z))=_{0}Y(\Psi(\xi, Y, Z)) \di \xi(Z)=_{0}\xi(Y)], 
\ee
and apply $\PFTPA$ to obtain standard such $\Psi$, which implies $\eqref{EXT}_{3}^{\st}$. 

\medskip

Finally, to establish the non-implication, a proof of $(\exists^{\st} T^{3})(\forall^{\st}\varphi^{2})\SO(T, \varphi)\di \SOT$ in $\B_{0}^{-}+\QFAC^{2,1}$ yields,
in the same way as in the proof of Corollary \ref{crossref}, a term of G\"odel's $T$ which computes the functional from $(U^{3})$ in terms of the functional from $(\E_{2})$. 
As established in Corollary \ref{crossref}, this is impossible.  
\end{proof}
In light of the previous two theorems, it again seems reasonable to include $\PFTPA$ in the base theory.  
Similar results should exist for the Suslin functional.  

\medskip

We could study the equivalence $(\mu^{2})\asa \UWKL$ from \cite{kohlenbach7}*{Cor.\ 3.5} where the latter states the existence of a functional $\Phi^{1\di 1}$ which outputs a path $\Phi(T)\in T$ on input any infinite binary tree $T$.  
This equivalence depends on the presence of the axiom of extensionality in that in systems without the latter, $\UWKL$ is not stronger than $\WKL$ itself (See \cite{kohlenbach7}*{Theorem 3.2}).  Hence, similar to the previous theorems, $\PFTPA$ is needed to prove equivalences between 
nonstandard versions of $\UWKL$ and $(\mu^{2})$, which provides another argument in favour of $\B_{0}$.   
In conclusion, to guarantee natural equivalences like in Theorem \ref{hanker} and to avoid constantly keeping track of extensionality, $\PFTPA$ seems unavoidable in the RM of Nonstandard Analysis.  
 
\medskip  

We finish this section with a remark on the role of $\QFAC$ regarding our results.  
\begin{rem}[On the role of $\QFAC$] \label{omdageaandringt}\rm
As noted at the beginning of this section, many results obtained above seem to crucially depend on fragments of $\QFAC$.
Hence, we should discuss the role of $\QFAC$ and its effect on RM.  

\medskip

First of all, we point out a mistake identified by the referee:  the correct statement of \cite{yama1}*{Theorem 2.2}, based on \cite{avi2}, is: $\RCAo+\QFAC^{0,1}+(\exists^{2})$ is a $\Pi_{2}^{1}$-conservative extension 
of $\ACA_{0}$.  Indeed, the restriction to $\QFAC^{0,1}$ is necessary for the elimination of extensionality procedure, as explained in the proof of \cite{kohlenbach7}*{Theorem 3.7}.  

\medskip

Secondly, the first result in Theorem \ref{hanker} depends on quantifier-free choice; indeed, $[\B_{0}\setminus \QFAC^{1,0}]+(\exists^{2})$ cannot prove $(\mu^{2})$ by Corollary \ref{conservativity2}, as $\textup{\eefa$^{\omega}$}+(\exists^{2})$ has a model consisting of majorisable functionals (See the proof of Theorem~\ref{hanker}), while $\mu^{2}$ is not majorisable. 
Similar results hold for Corollaries \ref{horkings} and \ref{crossref}.

\medskip

Thirdly, the proofs of Theorems \ref{hollfy} and \ref{polipo} suggest that to obtain a fragment of standard extensionality \eqref{EXT}$^{\st}$ via $\PFTPA$, we need a fragment of $\QFAC$. 
By the first part of Theorem~\ref{hollfy}, $\eqref{EXT}_{2}^{\st}$ is needed to obtain a nonstandard version of $(\mu^{2})$, namely $\paai$, from a nonstandard version of $(\exists^{2})$, namely $(\exists^{\st}\varphi^{2})(\forall^{\st}f^{1})\TJ(f, \varphi)$.  
Given the difference between $(\exists^{2})$ and $(\mu^{2})$ noted above, we believe $\eqref{EXT}_{2}^{\st}$ and $\QFAC^{2,0}$ are essential in obtaining the first and second part of Theorem~\ref{hollfy}. 

\medskip

Fourth, the use of $\QFAC^{2,1}$ in Theorem \ref{crofu} seems essential as in the second paragraph of this remark.  Indeed,  $\textup{\eefa$^{\omega}$}+\WKL$ has a model consisting of majorisable functionals, while $\Phi$ from the proof of Theorem \ref{crofu} does not seem majorisable. 
Similar reasoning applies to Theorem \ref{crofu2}, as $\Phi$ from \eqref{gow1} does not seem majorisable due to the occurrence of $g^{1}$ as part of $\Phi(X, f)$. 
 
\medskip

Finally, we do not know whether, but believe that, the use of fragments of $\QFAC$ is essential in Theorems \ref{massiveBT}, \ref{coreBT}, and \ref{holBT}.  
%
%
\end{rem}
\subsection{Equivalences between internal and external theorems}\label{kext}
As suggested by the title, we prove the equivalence between certain internal theorems, i.e.\ not involving Nonstandard Analysis, and externals ones, i.e.\ involving Nonstandard Analysis.  
In particular, we prove the following equivalence over various extensions of $\B_{0}$: $\WKL\asa \WKL^{\st}$, $\paai\asa (\exists^{2})$, $\ATR_{0}\asa \ATR_{0}^{\st}$, $\Paai\asa (S^{2})$,  
that $\STP$ is equivalent to the special fan functional from \cite{samGH}, and $\MUC\asa \NUC$.  

\medskip

Some of the aforementioned results rely on rather strong base theories, and we believe this to be unavoidable.  Note that a small number of equivalences in Reverse Mathematics are known to require a base theory \emph{stronger} than $\RCA_{0}$ and Hirschfeldt has asked whether there are more such equivalences (See \cite{montahue}*{\S6.1}).   

\medskip

Whenever possible, we show that $\PFTPA$ is essential for the equivalence at hand. Thus, we again observe that to guarantee natural equivalences like in Theorem \ref{hugger}, $\PFTPA$ seems unavoidable in the RM of Nonstandard Analysis.

\subsubsection{Weak K\"onig's lemma}
Firstly, we prove that \emph{weak K\"onig's lemma} is equivalent to itself relative to `st'.  This lemma states that a infinite binary tree has a path, and is abbreviated $\WKL$.  We take `$T^{1}\leq_{1}1$' to mean that $T^{1}$ is a binary tree and reserve this variable for this purpose.       
\begin{thm}\label{crofu}
The system $\B_{0}+\QFAC^{2,1}$ proves $\WKL\asa \WKL^{\st}$.  
\end{thm}
\begin{proof}
The contraposition of $\WKL$ (known as the \emph{fan theorem}) is:
\[
(\forall T\leq_{1}1)\big[ (\forall \alpha\leq_{1}1)(\exists n^{0})(\overline{\alpha}n\not \in T) \di  (\exists k^{0})(\forall \beta^{0})(\exists i\leq k)(|\beta|= k \di \overline{\beta}i\not\in T) \big],
\]
which immediately yields (by strengthening the antecedent)
\[
(\forall T\leq_{1}1, g^{2})\big[ (\forall \alpha\leq_{1}1)(\overline{\alpha}g(\alpha)\not \in T) \di  (\exists k^{0})(\forall \beta^{0})(\exists i\leq k)(|\beta|= k \di \overline{\beta}i\not\in T) \big],
\]
and bringing all unbounded quantifiers outside, we obtain:
\[
(\forall T\leq_{1}1, g^{2})(\exists \alpha\leq_{1}1,k^{0})\big[ (\overline{\alpha}g(\alpha)\not \in T) \di (\forall \beta^{0})(\exists i\leq k)(|\beta|= k \di \overline{\beta}i\not\in T) \big],
\]
where the formula in square brackets is equivalent to a quantifier-free one.  Applying $\QFAC^{2,1}$, we obtain $\Phi$ producing $\alpha, k$ from input $T, g$.  
By $\PFTPA$, we may assume such $\Phi$ is standard, which implies
\[
(\forall^{\st} T\leq_{1}1, g^{2})(\exists^{\st} \alpha\leq_{1}1,k^{0})\big[ (\overline{\alpha}g(\alpha)\not \in T) \di (\forall \beta^{0})(\exists i\leq k)(|\beta|= k \di \overline{\beta}i\not\in T) \big],
\]
as standard objects produce standard output for standard input.  Hence, we have established the forward implication, and the reverse one follows by dropping the second `st' in the previous sentence and applying $\PFTPA$.    
\end{proof}
We do not know if $\PFTPA$ is essential in the previous.  Note that the same result can be proved for \emph{weak weak K\"onig's lemma} $\WWKL$ from \cite{simpson2}*{X.1}, and a host of statements from the Reverse Mathematics zoo (\cite{damirzoo}) with similar syntactic form, namely of the form $(\forall x^{\rho})(\exists y^{\tau})A_{0}(x, y)$ with $A_{0}$ a quantifier-free formula.   

\subsubsection{Arithmetical comprehension}
Secondly, we prove $\paai\asa (\mu^{2})$ over a conservative extension of $\B_{0}+\WKL$.  
To this end, we consider Kohlenbach's generalisations of $\WKL_{0}$ as introduced in \cite{kohlenbach4}*{\S5-6}.  
These generalisations have been studied in \cite{dagsam} and give rise to the following functional and fragment of \emph{Transfer}.  
\be\label{KOT}\tag{$\kappa^3$}
(\exists \kappa^{2\di 1})(\forall Y^{2})\big[ (\exists f^{1}\leq_{1}1)(Y(f)=0) \di  [Y(\kappa(Y))=0 \wedge \kappa(Y)\leq_{1}1]\big].
\ee
\be\tag{\textsf{\textup{WT}}}
(\forall^{\st} Y^{2})\big[ (\exists f^{1}\leq_{1}1)(Y(f)=0)\di(\exists^{\st} f^{1}\leq_{1}1)(Y(f)=0)  \big]
\ee
Note that $\RCAo+(\kappa^{3})$ is conservative over $\WKL_{0}$ since the intuitionistic fan functional $\textsf{MUC}$ readily implies $(\kappa^{3})$ (See \cite{kohlenbach2}*{\S3}).  
As such, $\B_{0}+(\kappa^{3})$ is a reasonable base theory according to Hirschfeldt's criteria from \cite{dslice}*{p.\ 13}.  On the other hand, $(\kappa^{3})$ has some `hidden power' compared to $\WKL_{0}$: Kohlenbach has shown in a private communication\footnote{The proof amounts to the observation that $\N^\N$ is recursively homeomorphic to a $\Pi^0_2$-subset of Cantor space. Since this set is computable in $\exists^{2}$, any oracle call to $\exists^{3}$ (aka $\E_{2}$) can be rewritten to an equivalent oracle call to $\kappa^{3}$, in a uniform way.  
See also \cite{dagsam}*{\S6}.\label{dagwrotethis}} that $\RCAo$ proves $[(\kappa^{3})+(\exists^{2})]\di (\E_{2})$, where the latter implies full second-order arithmetic, and hence \emph{dwarfs} the other functionals.    
\begin{thm}\label{hugger}
$\B_{0}+(\kappa^{3})$ proves $(\mu^{2})\asa \paai$; $\B_{0}^{-}+(\kappa^{3})$ does not.  
\end{thm}
\begin{proof}
The non-implication follows by noting that applying Theorem \ref{consresultcor} to `$\B_{0}^{-}+(\kappa^{3})+(\mu^{2})\vdash \paai$' yields a term $t$ in G\"odel's $T$ such that $\MU(t)$, which is impossible.  
The forward implication is immediate by applying $\PFTPA$ to $(\mu^{2})$ and noting that standard functionals have standard output for standard input.  

\medskip

For the reverse implication, note that $\paai$ implies $(\mu^{2})^{\st}$ (using $\HACint$) and that $(\kappa^{3})$ implies $\textsf{WT}$ (using $\PFTPA$).  
Now, $(\mu^{2})^{\st}$ implies 
\be\label{bezig}
(\exists^{\st}\mu^{2})(\forall^{\st}n^{0})\underline{(\forall^{\st}f^{1}\leq_{1}1)\big[f(n)=0 \di f(\mu(f))=0\big]}, 
\ee
and apply (the contraposition of ) $\textsf{WT}$ to the underlined formula in \eqref{bezig} to obtain
\be\label{forju}
(\exists^{\st}\mu^{2})(\forall^{\st}n^{0}){(\forall f^{1}\leq_{1}1)\big[f(n)=0 \di f(\mu(f))=0\big]}, 
\ee
Let $Z(f, n)$ be the \emph{standard} characteristic functional of the formula in square brackets in \eqref{forju} and note that \eqref{forju} implies the following:
\be\label{forju2}
(\exists^{\st}\mu^{2})(\forall^{\st}n^{0})(\forall f^{1}\leq_{1}1)(Z(f, n)=1), 
\ee
and using (standard thanks to $\PFTPA$) $\kappa^{3}$ as in $(\kappa^{3})$, we obtain 
\be\label{forju3}
(\exists^{\st}\mu^{2})(\forall^{\st}n^{0})(Z(\kappa(\lambda f.Z(f, n)), n)=1), 
\ee
which immediately yields the following thanks to $\paai$:
\be\label{forju4}
(\exists^{\st}\mu^{2})(\forall n^{0})(Z(\kappa(\lambda f.Z(f, n)), n)=1), 
\ee
which by the definition of $\kappa$ is equivalent to
\be\label{forju5}
(\exists^{\st}\mu^{2})(\forall n^{0})(\forall f\leq_{1}1)(Z(f, n)=1), 
\ee
and finally we obtain, by the definition of $Z$, that 
\be\label{forju6}
(\exists^{\st}\mu^{2})(\forall n^{0})(\forall f\leq_{1}1)\big[f(n)=0 \di f(\mu(f))=0\big], 
\ee
which is essentially $(\mu^{2})$ thanks to $\PFTPA$, and we are done.
\end{proof}
It is an interesting question whether we can weaken the base theory in the previous theorem to other conservative extensions of $\WKL_{0}$.  
Another interesting question is which system can prove $\textsf{WT}\asa (\kappa^{3})$?  It is straightforward to show $\B_{0}+(\mu^{2})\vdash [\textsf{WT}\asa \SOT]$ based on the aforementioned result by Kohlenbach, but then obtaining $(\kappa^{3})$ seems impossible.   
\subsubsection{Transfinite recursion}
Thirdly,  we prove that \emph{arithmetical transfinite recursion} is equivalent to itself relative to `st'.  
Regarding definitions, the system $\ATR_{0}$ is $\ACA_{0}$ plus the second-order schema of arithmetical transfinite recursion:
\be\tag{$\ATR_{\theta}$}
(\forall X^{1})\big[\WO(X)\di (\exists Y^{1})H_{\theta}(X, Y) \big], 
\ee
for any arithmetical $\theta$.  Note that $\WO(X)$ expresses that $X$ is a countable well-ordering and $H_{\theta}(X, Y)$ expresses that $Y$ is the result from iterating $\theta$ along $X$.  
More details and related results may be found in \cite{simpson2}*{V.2}.  
To avoid quantifying over (arithmetical) formulas, we shall make use of an equivalent\footnote{By considering finite well-orders, $\ATR$ implies comprehension for formulas of the form $(\exists m^{0})(f(m,n, \overline{Z}m)=0)$.  By \emph{Kleene's normal form theorem} in $\RCA_{0}$ (\cite{simpson2}*{II.2.7}), a $\Sigma_{1}^{0}$-formula can be formulated as in the latter form.  
Since comprehension for $\Sigma_{1}^{0}$-formulas implies $\ACA_{0}$ by \cite{simpson2}*{III.1.3}, $\ATR\di\ACA_{0}$, and the latter allows the replacement of arithmetical formulas by quantifier-free ones.  Hence, $\ATR_{0}\di\ATR$ over $\RCA_{0}$; the other direction is immediate.} reformulation of arithmetical transfinite recursion.  
To this end, let $H_{f}(X, Y)$ be $H_{\theta}(X, Y)$ as above for $\theta(n,Z)\equiv(\exists m^{0})(f(m,n, \overline{Z}m)=0)$.  We define $\ATR$ as follows:  
\be\tag{$\ATR$}
(\forall X^{1},f^{1})\big[\WO(X)\di (\exists Y^{1})H_{f}(X, Y) \big].  
\ee
With the previous notation, we have the following theorem.  
\begin{thm}\label{crofu2}
The system $\B_{0}+(\exists^{2})+\QFAC^{1,1}$ proves $\ATR\asa \ATR^{\st}$.  
\end{thm}
\begin{proof}
We first show that $\paai$ is available as follows: applying $\PFTPA$ to $(\exists^{2})$, there is a standard $\varphi^{2}$ such that $\TJ(\varphi)$.  
Hence, we obtain $(\exists^{\st}\mu^{2})\MU(\mu)$ by Theorem~\ref{hanker}, and the latter axiom implies $\paai$ by \eqref{hankio}.  

\medskip

For the reverse implication, let $\WO(X, g)$ be the formula expressing that $g^{1}$ is not an infinite descending sequence in $\leq_{X}$.  Then $\ATR^{\st}$ implies
\be\label{gow}
(\forall^{\st} X^{1},f^{1})(\exists^{\st} g^{1}, Y^{1})\big[\WO(X,g)\di H_{f}(X, Y) \big], 
\ee
as the formula in square brackets is arithmetical and $\paai$ is given thanks to $(\exists^{2})$, Theorem \ref{hanker} and \eqref{hankio}.  
Dropping the second `st' in the sentence \eqref{gow} and applying $\PFTPA$ yields $\ATR$.  For the forward implication, $\ATR$ readily implies
\[
(\forall X^{1},f^{1})(\exists g^{1}, Y^{1})\big[\WO(X,g)\di H_{f}(X, Y) \big], 
\]
and the combination of $\QFAC^{1,1}$ and $(\exists^{2})$ yields 
\be\label{gow1}
(\exists \Phi)(\forall X^{1},f^{1})(\exists g^{1}, Y^{1}\in \Phi(X, f))\big[\WO(X,g)\di H_{f}(X, Y) \big], 
\ee
and by $\PFTPA$ there is standard such $\Phi$.  Since standard functionals produce standard output for standard input, we obtain \eqref{gow}, and $\ATR^{\st}$ via $\paai$.
\end{proof}
Note that the previous theorem immediately generalises to \emph{any} sentence of the form $(\forall x^{\tau})(\exists y^{\sigma})A_{1}(x, y)$ with $A_{1}$ arithmetical.

\subsubsection{Comprehension beyond the arithmetical}
Fourth, we prove results analogous to Theorem \ref{hugger} for $(\mu_{1})$ and $(\E_{2})$.  
We could generalise the proof of the former theorem to the latter functionals, but a more straightforward `trick' based on Theorem~\ref{hugger} turns out to be sufficient (and much neater).  
We therefore formulate the theorems in this section as corollaries (to the latter theorem).

\medskip

First of all, we prove that the functional $(\mu_{1})$, which essentially expresses $\Pi_{1}^{1}$-comprehension, is equivalent to $\Paai$ as follows:
\be\tag{$\Paai$}
(\forall^{\st}f^{1})\big[ (\exists g^{1})(\forall x^{0})(f(\overline{g}x)=0)\di (\exists^{\st}g^{1})(\forall x^{0})(f(\overline{g}x)=0)\big], 
\ee
which is the nonstandard counterpart of $\FIVE$.  We have the following theorem.  
\begin{cor}
$\B_{0}+(\kappa^{3})$ proves $(\mu_{1})\asa \Paai$; $\B_{0}^{-}+(\kappa^{3})$ does not.  
\end{cor}
\begin{proof}
The non-implication follows as in Theorem \ref{hugger}, while the forward implication is immediate by applying $\PFTPA$ to $(\mu_{1})$ and noting that standard functionals have standard output for standard input.  
For the reverse implication, note that $\Paai$ implies $\paai$ and hence $(\mu^{2})$ by Theorem \ref{hugger}.  
As noted in the previous section (Footnote \ref{dagwrotethis} in particular), $\RCAo$ proves $[(\kappa^{3})+ (\exists^{2})]\di (\E_{2})$.
The latter functional clearly implies $(\mu_{1})$ and we are done.  
\end{proof}
Secondly, we prove that the functional $(\E_{2})$, which essentially expresses full second-order comprehension, is equivalent to $\SOT$.  
\begin{cor}
$\B_{0}+(\kappa^{3})$ proves $(\E_{2})\asa \SOT$; $\B_{0}^{-}+(\kappa^{3})$ does not.  
\end{cor}
\begin{proof}
The non-implication follows as in Theorem \ref{hugger}, while the forward implication is immediate by applying $\PFTPA$ to $(U_{2})$ and noting that standard functionals have standard output for standard input.  
For the reverse implication, note that $\SOT$ implies $\paai$ and hence $(\mu^{2})$ by Theorem \ref{hugger}.  
As noted in the previous section (Footnote \ref{dagwrotethis} in particular), $\RCAo$ proves $[(\kappa^{3})+ (\exists^{2})]\di (\E_{2})$.
\end{proof}
In conclusion, the previous suggests that any fragment of $\SOT$ which implies $\paai$ is equivalent to the comprehension axiom associated to this fragment.  
However, having $(\kappa^{3})$ in the base theory is not satisfactory, as we also obtain `superfluous' equivalences likes $(\exists^{2})\asa (\mu_{1})\asa (\E_{2})$.  
Hence, it is an interesting question how we can weaken $(\kappa^{3})$ to make sure these `superfluous' equivalences do not occur.  

\subsubsection{Nonstandard compactness}
Fifth, we prove an equivalence involving the nonstandard compactness of Cantor space and a conservative extension of $\WKL$, namely the \emph{special fan functional}, introduced in \cite{samGH} and studied in detail in \cite{dagsam}.   
\bdefi[Special fan functional]\label{dodier}
We define $\SCF(\Theta)$ as follows for $\Theta^{(2\di (0\times 1))}$:
\[
(\forall g^{2}, T^{1}\leq_{1}1)\big[(\forall \alpha \in \Theta(g)(2))  (\overline{\alpha}g(\alpha)\not\in T)
\di(\forall \beta\leq_{1}1)(\exists i\leq \Theta(g)(1))(\overline{\beta}i\not\in T) \big]. 
\]
Any functional $\Theta$ satisfying $\SCF(\Theta)$ is referred to as a \emph{special fan functional}.
\edefi
From a computability theoretic perspective, the main property of the special fan functional $\Theta$ is the selection of $\Theta(g)(2)$ as a finite sequence of binary sequences $\langle f_0 , \dots, f_n\rangle $ such that the neighbourhoods defined from $\overline{f_i}g(f_i)$ for $i\leq n$ form a cover of Cantor space;  almost as a by-product, $\Theta(g)(1)$ can then be chosen to be the maximal value of $g(f_i) + 1$ for $i\leq n$. We stress that $g^{2}$ in $\SCF(\Theta)$ may be \emph{discontinuous} and that Kohlenbach has argued for the study of discontinuous functionals in higher-order RM (See \cite{kohlenbach2}).  It is known that $\RCAo+(\exists \Theta)\SCF(\Theta)$ is conservative over $\WKL_{0}$ (See \cite{dagsam} or \cite{samGH}).  
Furthermore, the special fan functional naturally emerges from Tao's notion of \emph{metastability}, as discussed in \cite{sambrouw, samflo}.

\medskip

The special fan functional was originally derived from $\STP$, the \emph{nonstandard compactness} of Cantor space as in \emph{Robinson's theorem} (\cite{loeb1}*{}).  
This fragment of \emph{Standard Part} is also known as the `nonstandard counterpart of weak K\"onig's lemma' (\cite{keisler1}).  
\be\tag{$\STP$}
(\forall \alpha^{1}\leq_{1}1)(\exists^{\st}\beta^{1}\leq_{1}1)(\alpha\approx_{1}\beta),
\ee  
as explained by the equivalence between $\STP$ and \eqref{fanns}.  It is well-known that the axioms \emph{Transfer} and \emph{Standard Part} of Nelson's \textsf{IST} are independent (\cite{reeken}), and 
the same property apparently holds for the fragments $\Paai$ and $\STP$, in light of the non-implications in \eqref{nolinord}.  As a result, the Suslin functional does not suffice to prove the existence of the special fan functional, as discussed in Footnote~\ref{janarjan}. 
\begin{thm}\label{lapdog}
In $\B_{0}^{-}$, $\STP$ is equivalent to the following:
\begin{align}\label{frukkklk}
(\forall^{\st}g^{2})(\exists^{\st}w^{1}\leq_{1}1, k^{0})\big[(\forall T^{1}\leq_{1}1)\big( & (\forall \alpha^{1} \in w)(\overline{\alpha}g(\alpha)\not\in T)\\
&\di(\forall \beta\leq_{1}1)(\exists i\leq k)(\overline{\beta}i\not\in T)\big) \big], \notag
\end{align}  
as well as to the following:
\begin{align}\label{fanns}
(\forall T^{1}\leq_{1}1)\big[(\forall^{\st}n)(\exists \beta^{0})&(|\beta|=n \wedge \beta\in T ) \di (\exists^{\st}\alpha^{1}\leq_{1}1)(\forall^{\st}n^{0})(\overline{\alpha}n\in T)   \big].
\end{align}
Furthermore, $\B_{0}^{-}$ proves $(\exists^{\st}\Theta)\SCF(\Theta)\di \STP$.
\end{thm}
\begin{proof}
A detailed proof may be found in \cite{dagsam} or \cite{samGH}.  In a nutshell, the implication \eqref{frukkklk}$\leftarrow$\eqref{fanns} follows by taking the contraposition of the latter and introducing standard $g^{2}$ in the antecedent of the resulting formula.  
One then uses \emph{Idealisation} \textsf{I} to pull the standard quantifiers to the front and obtains \eqref{frukkklk}.  The other implication follows by pushing the standard quantifiers in the latter back inside.   
For the remaining implication $\STP\di \eqref{fanns}$ (the other one and the final part then being trivial), one uses \emph{overspill} (See \cite{benno}*{\S3}) to obtain a sequence of nonstandard length for a tree $T\leq_{1}1$ satisfying the antecedent of  \eqref{fanns}, and $\STP$ converts this sequence into a standard path in $T$.  
\end{proof}
As shown in \cite{dagsam}, the special fan functional can be computed (Kleene S1-S9 as in \cite{longmann}*{\S5.1.1}) in terms of $(\E_{2})$, but it cannot be computed by $(\mu^{2})$ or $(S^{2})$, \emph{or any type two functional}.
Similarly, while the base theory in the following theorem is strong, it cannot\footnote{If $\B_{0}+(S^{2})+\QFAC^{2,1}$ proves $\STP$, then by Theorems \ref{conservativity2} and \ref{massiveBT}, $\RCAo+(S^{2})+\QFAC$ proves $(\exists \Theta)\SCF(\Theta)$, but this is impossible by the results in \cite{dagsamIII}*{\S3}.\label{janarjan}} prove $\STP$ as suggested by the results in Section \ref{RM}.  
\begin{thm}\label{massiveBT}
The system $\B_{0}+(S^{2})+\QFAC^{2,1}$ proves $ \STP\asa (\exists \Theta)\SCF(\Theta)$, while the system $\B_{0}^{-}+(S^{2})+\QFAC$ does not.  
\end{thm}
\begin{proof}
The reverse implication is immediate using $\PFTPA$ and Theorem \ref{lapdog}.  
For the forward implication, $\STP$ implies \eqref{frukkklk} by the latter theorem.  Drop the second `st' in \eqref{frukkklk}, and apply $\PFTPA$ to the resulting formula to obtain
\begin{align}\label{frukkklk2}
(\forall g^{2})(\exists w^{1}\leq_{1}1, k^{0})\big((\forall T^{1}\leq_{1}1)\big[ & (\forall \alpha^{1} \in w)(\overline{\alpha}g(\alpha)\not\in T)\\
&\di(\forall \beta\leq_{1}1)(\exists i\leq k)(\overline{\beta}i\not\in T)\big] \big), \notag
\end{align} 
The formula in square brackets in \eqref{frukkklk2} is such that $T$ only occurs as $\overline{T}x$ where $x$ is at least the maximum of $k+1$ and $g(w(i))+1$ for $i<|w|$.  
Thus, let $A_{0}(g, w, k, T)$ be formula in square brackets in \eqref{frukkklk2} and $x_{0}$ the aforementioned maximum; we obtain $A_{0}(g, w, k, T)\asa (\exists x^{0}\geq x_{0})A_{0}(g, w, k, \overline{T}x)$.
In this light,  $(\forall T\leq_{1}1)A_{0}(g, w, k, T)$ has the right (equivalent) form to be decidable by the Suslin functional.  Thus, apply $\QFAC^{2,1}$ to \eqref{frukkklk2} to obtain $\Theta$ producing $w^{1}, k^{0}$ from $g^{2}$.  

\medskip

The non-implication follows from \cite{dagsam}*{Theorem 4.2} as the latter expresses that the special fan functional is not computable in any type two functional.  Indeed, $\STP$ is equivalent to \eqref{frukkklk} by Theorem \ref{lapdog} and applying Theorem \ref{consresultcor} to 
\[
\B_{0}^{-}+(S^{2})+\QFAC+(\exists \Theta)\SCF(\Theta)\vdash \eqref{frukkklk}, 
\]
one obtains a term $t$ of G\"odel's $T$ such that $\SCF(t)$, which is impossible.  
\end{proof}
Clearly, $(\kappa^{3})$ could be used in the proof instead of $(S^{2})$, but the former computes (Kleene S1-S9) the special fan functional while the latter does not (See \cite{dagsam}*{\S6.4}).  

\medskip

The same theorem is readily proved for \textsf{LMP}, the nonstandard counterpart of $\WWKL_{0}$ from \cite{pimpson}, and the associated `weak fan functional' $\Lambda$ from \cite{dagsam}*{\S3.3}.  

\medskip

Finally, we point out a result from \cite{samcie18}, namely an equivalence between $\STP$ and the Heine-Borel theorem \emph{in the general\footnote{The Heine-Borel theorem in RM is restricted to \emph{countable} covers (\cite{simpson2}*{IV.1}).} case}, i.e.\ the statement that any (possibly uncountable) 
open cover of the unit interval has a finite sub-cover.  In particular, any $\Psi:\R\di \R^{+}$ gives rise to a `canonical' open cover $\cup_{x\in [0,1]}I_{x}$ of $[0,1]$ where $I_{x}^{\Psi}\equiv (x-{\Psi(x)}, x+{\Psi(x)})$. 
Hence, the Heine-Borel theorem trivially implies: 
\be\label{zosimpelistnie}\tag{$\HBU$}\textstyle
(\forall \Psi:\R\di \R^{+})(\exists w^{1}){(\forall x\in [0,1])(\exists y\in w)(x\in I_{y}^{\Psi})}. 
\ee
As discussed in \cite{dagsamIII}*{\S1}, $\HBU$ is part of ordinary mathematics as it predates set theory.  
Furthermore, $\HBU$ is equivalent to many basic properties of the \emph{gauge integral} (\cite{dagsamIII}*{\S3.3}).  The latter is an extension of Lebesgue's integral and provides a (direct) formalisation of the 
Feyman path integral.  
\begin{thm}\label{coreBT}
The system $\B_{0}+(\exists^{2})+\QFAC^{2,1}$ proves that $\STP\asa \HBU$, while the system $\B_{0}^{-}+(\exists^3)+\QFAC$ does not.  
\end{thm}
\begin{proof}
See \cite{samcie18}*{\S3.2}.
\end{proof}
In conclusion, $\STP$ can be viewed as (a nonstandard version of) open-cover compactness, for \emph{uncountable} covers. 

\subsubsection{Nonstandard continuity}\label{Sectie9}
Sixth, we prove an equivalence involving the intuitionistic fan functional $\MUC$ and  nonstandard uniform continuity on Cantor space.  
\be\tag{$\MUC$}
(\exists \Omega^{3})(\forall Y^{2}) (\forall f, g\leq_{1}1)(\overline{f}\Omega(Y)=\overline{g}\Omega(Y)\notag \di Y(f)=Y(g)).  
\ee
Note that $\Omega$ computes a modulus of uniform continuity for every type two functional on Cantor space, i.e.\ $\MUC$ is inconsistent with classical mathematics as $(\exists^{2})$ is equivalent to the existence of discontinuous functionals (See \cite{kohlenbach2}*{\S3}).  The nonstandard counterpart of $\MUC$ is as follows:  
\be\tag{$\NUC$}
(\forall^{\st} Y^{2}) (\forall f, g\leq_{1}1)({f}\approx_{1}{g} \di Y(f)=Y(g)),   
\ee
which expresses that every standard type two functional is \emph{nonstandard} uniformly continuous on Cantor space.  Let $\MUC_{C}$ and $\NUC_{C}$ be the above principles with $Y^{2}$ restricted to `continuous' functionals $Y^{2}\in C$ as follows:   
\be\label{tooexplicit}
[Y^{2}\in C ] \equiv (\forall f^{1})(\exists N^{0})(\forall g^{1})(\overline{f}N=\overline{g}N \di Y(f)=Y(g)).
\ee
As noted above, the intuitionistic fan functional computes the special one, as well as $\NUC\di \STP$ (See \cite{samGH}*{\S3}), but the RM of $\MUC$ is much more elegant than the RM of the special fan functional in Theorem \ref{massiveBT}.
\begin{thm}\label{holBT}
The system $\B_{0}+\QFAC^{2,0}$ proves $ [(\kappa^{3})+\NUC]\asa \MUC$, while the system $\B_{0}^{-}+\QFAC$ does not prove $\MUC_{C}\di \NUC_{C}$.  
\end{thm}
\begin{proof}
The reverse implication follows by applying $\PFTPA$ to $\MUC$ and noting that for standard $\Omega^{3}$, $\Omega(Y)$ is standard for standard $Y^{2}$, and hence the latter is nonstandard uniformly continuous on Cantor space as in $\NUC$.  Furthermore, $\MUC$ clearly implies $(\kappa^{3})$.    
For the forward implication, assume $\NUC$ and resolve `$\approx_{1}$':
\[
(\forall^{\st} Y^{2}) (\forall f, g\leq_{1}1)((\forall^{\st}N)(\overline{f}N=\overline{g}N) \di Y(f)=Y(g)),  
\]
and bringing the standard quantifier outsideas far as possible, we obtain:
\[
(\forall^{\st} Y^{2}) (\forall f, g\leq_{1}1)(\exists^{\st}N)(\overline{f}N=\overline{g}N \di Y(f)=Y(g)).  
\]
Applying (the contraposition of) \emph{Idealisation} to the previous formula, we obtain 
\be\label{kuklokp}
(\forall^{\st} Y^{2})(\exists^{\st}M) (\forall f, g\leq_{1}1)(\exists N\leq M)(\overline{f}N=\overline{g}N \di Y(f)=Y(g)), 
\ee
which immediately yields the following by dropping one `\st':
\[
(\forall^{\st} Y^{2})(\exists M) (\forall f, g\leq_{1}1)(\overline{f}M=\overline{g}M \di Y(f)=Y(g)), 
\] 
and applying $\PFTPA$ finally brings us to:
\be\label{huraaah}
(\forall Y^{2})(\exists M^{0}) \underline{(\forall f, g\leq_{1}1)(\overline{f}M=\overline{g}M \di Y(f)=Y(g))}.  
\ee
Thanks to $(\kappa^{3})$, the underlined formula may be treated as quantifier-free;  applying $\QFAC^{2,0}$ to \eqref{huraaah} now yields $\MUC$.  

\medskip

The non-implication follows from \cite{noortje}*{Theorem 4.40} as the latter theorem implies that the (classical) fan functional is not given by a term from G\"odel's $T$.  
Indeed, $\NUC$ implies \eqref{kuklokp} and $\NUC_{C}$ similarly implies 
\be\label{koormaker}
(\forall^{\st} Y^{2}\in C)(\exists^{\st}M) (\forall f, g\leq_{1}1)(\exists N\leq M)(\overline{f}N=\overline{g}N \di Y(f)=Y(g)), 
\ee
as `$Y^{2}\in C$' is an internal formula by \eqref{tooexplicit}.  
However, any proof of the form
\be\label{tering}
\B_{0}^{-}+\MUC_{C}+\QFAC\vdash \eqref{koormaker}
\ee
yields a term of G\"odel's $T$ which computes the classical fan functional by applying Theorem \ref{consresultcor} to \eqref{tering}.  
\end{proof}
We do not know if $\PFTPA$ is essential in the equivalence, but do expect so.
Note that we avoided the use of the non-classical $\MUC$ in the non-implication in the theorem.  This is due to our use of \cite{noortje}*{Theorem 4.40} in the previous proof: the latter theorem is part of classical mathematics, and hence does not apply to $\MUC$.

\begin{rem}[On mathematical naturalness]\rm
We have proved the equivalence between respectively the comprehension\footnote{Note that $(\exists^{2})$ and $(S^{2})$ constitute comprehension principles; these are respectively equivalent to $(\mu^{2})$ and $(\mu_{1})$, modulo a dash of choice, namely $\QFAC^{1,1}$.} principles $(\mu^{2})$ and $(\mu_{1})$, and the Transfer principle limited respectively to $\Pi_{1}^{0}$ and $\Pi_{1}^{1}$-formulas, and this over a weak base theory based on $\efa+\WKL$.
Now, $(\exists^{2})$ and $(S^{2})$ correspond to two of the `Big Five systems', which leads us to the
following foundational claim regarding `mathematical naturalness' of logical systems, as qualified by Simpson in \cite{simpson2}*{I.12}:
\begin{quote}
From the above it is clear that the five basic systems $\RCA_{0}$, $\WKL_{0}$, $\ACA_{0}$, $\ATR_{0}$, $\FIVE$ arise naturally from investigations of the Main Question. The proof that these systems are mathematically natural is provided by Reverse Mathematics.
\end{quote}
Given the equivalences proved above, it would seem that $\paai$ and $\Paai$ also count as mathematically natural, in direct contrast to the view certain people\footnote{As discussed in \cite{samsynt}, both Alain Connes and Errett Bishop have expressed extremely negative (but unfounded) opinions of Nonstandard Analysis.} hold regarding Nonstandard Analysis.
In other words, \emph{Nonstandard Analysis never looked so standard}.
\end{rem}
\begin{ack}\rm
We are grateful to the anonymous referee for the many helpful suggestions, including the observation that $(\exists^{2})$ could be omitted from the base theory in Corollary \ref{horkings}, as well as the correction to \cite{yama1}*{Theorem 2.2} from Remark \ref{omdageaandringt}.
\end{ack}

\bye